\newif\ifpictures
\picturestrue 

\documentclass[12pt]{amsart}

\usepackage{amssymb,amsmath}
\usepackage{paralist}

\usepackage{lmodern} 
\usepackage{microtype}

\usepackage[a4paper, hmargin={2cm}]{geometry}

\usepackage{amsthm}
\usepackage{dsfont}
\usepackage{tikz}

\usepackage[english]{babel} %
\usepackage{csquotes} %
\usepackage{xcolor}

\usepackage[colorlinks=true, citecolor=blue, urlcolor=violet, breaklinks=true]{hyperref}

\numberwithin{equation}{section}
\newtheorem{theorem}{Theorem}[section]
\newtheorem{thm}[theorem]{Theorem}
\newtheorem{proposition}[theorem]{Proposition}
\newtheorem{lemma}[theorem]{Lemma}
\newtheorem{corollary}[theorem]{Corollary}

\theoremstyle{definition}
\newtheorem{definition}[theorem]{Definition}
\newtheorem{example}[theorem]{Example}
\newtheorem{remark}[theorem]{Remark}

\DeclareMathOperator{\conv}{conv}
\DeclareMathOperator{\cone}{cone}
\DeclareMathOperator{\interior}{int}
\DeclareMathOperator{\relint}{relint}

\DeclareMathOperator{\supp}{supp}

\DeclareMathOperator{\rec}{rec}
\DeclareMathOperator{\inter}{int} %
\DeclareMathOperator{\pos}{pos} %

\newcommand{\zerob}{\mathbf{0}}
\newcommand{\oneb}{\mathbf{1}}

\newcommand{\N}{\mathbb{N}}
\newcommand{\R}{\mathbb{R}}

\newcommand{\cA}{\mathcal{A}}

\newcommand{\Nsf}{\mathsf{N}}

\author{Helen Naumann}
\author{Thorsten Theobald}

\address{Helen Naumann, Thorsten Theobald:
	Goethe-Universit\"at, FB 12 -- Institut f\"ur Mathematik,
	Postfach 11 19 32, D--60054 Frankfurt am Main, Germany}

\email{\{naumann,theobald\}@math.uni-frankfurt.de}

\title[]{Sublinear circuits for polyhedral sets}

\date{\today}

\begin{document}
\begin{abstract}
Sublinear circuits are generalizations of the affine circuits in matroid theory,
and they arise as the convex-combinatorial core
underlying constrained non-negativity 
certificates of exponential sums and of polynomials based on the 
arithmetic-geometric inequality. Here, we study the polyhedral combinatorics
of sublinear circuits for polyhedral constraint sets.

We give results on the relation between the sublinear circuits and their supports
and provide necessary as well as sufficient criteria for sublinear circuits.
Based on these characterizations, we provide some explicit results and enumerations
for two prominent polyhedral cases, namely the non-negative orthant and the 
cube $[-1,1]^n$.
\end{abstract}

   \maketitle

\section{Introduction}

Let $\cA$ be a non-empty finite subset of $\R^n$ and
$\R^{\cA}$ denote the set of real vectors 
whose components are indexed by the set $\mathcal{A}$.

For $\beta \in \cA$, write
\begin{equation}\label{eq:nbeta}
  N_{\beta} = \left\{ \nu \in \R^{\cA} \,:\, \nu_{\setminus \beta} \geq 
  \zerob,~ \sum_{\alpha \in \cA} \nu_{\alpha} = 0 \right\}
\end{equation}
for the cone of vectors in $\R^{\cA}$ whose entries sum to~0 and which may only have a negative entry in 
component $\beta$. Here, 
 $\nu_{\setminus \beta}$ abbreviates the vector in $\R^{\cA \setminus \{\beta\}}$
 which consists of all the components of $\nu$ except the component indexed with 
 $\beta$. 
 
In the context of non-negative polynomials and non-negative exponential sums, Murray
and the authors \cite{mnt-2020} have recently introduced the following 
generalization of the simplicial circuits of an affine matroid.
A non-zero vector $\nu^* \in N_\beta$ is called a \textit{sublinear circuit of 
$\mathcal{A}$ with respect to a given convex set $X$} (for short, \emph{$X$-circuit}) if
\begin{enumerate}
\item $\sup_{x \in X} ((-\cA \nu^*)^T x) < \infty$,
\item if $\nu \mapsto {\sup_{x \in X} ((-\cA \nu)^T x) }$ is linear on a two-dimensional cone in $N_{\beta}$, then $\nu^*$ is not in the relative interior of that cone.

\end{enumerate}
Here, $\cA$ is treated as a linear operator $\cA:\R^\cA\rightarrow \R^n$, $\nu\mapsto\sum_{\alpha\in\cA} \alpha\nu_\alpha$.

In the special case $X=\R^n$, the first condition is equivalent 
to $\cA \nu^* = \mathbf{0}$, which together with the second condition 
tells us that $\nu^*$ is a circuit of the affine
matroid with ground set $\mathcal{A} \subset \R^n$
(see, for example, \cite{FdW-2019,oxley-book}). Note that these 
$\R^n$-circuits are uniquely determined (up to scaling) by their supports. Moreover, the 
condition $\nu^* \in N_\beta$ enforces that the convex hull of the support 
$\supp\nu^* := \{\alpha \,:\, \nu^{*}_{\alpha} \neq 0 \}$ forms a simplex (possibly of dimension less than $n$) and exactly one element in $\supp\nu^*$ 
is contained in the relative interior of this simplex.
By the identification of circuits $\nu$ with their supports,
it is customary to call a subset $A \subset \cA$ a \emph{simplicial circuit} if its convex hull
$\conv(A)$ forms a simplex and the relative interior $\relint \conv(A)$ contains 
exactly one element of $A$. See Figure~\ref{fi:circuit1}.

\ifpictures
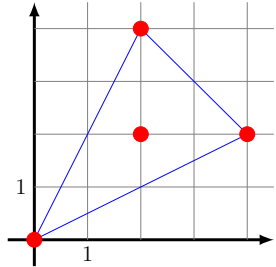
\begin{figure}[h]
\begin{tikzpicture}

  \useasboundingbox (-0.5,-0.5) rectangle (5,4);
    \scope[transform canvas={scale=.7}]
    \coordinate (Origin)   at (0,0);
    \coordinate (xaxismin) at (-0.5,0);
    \coordinate (xaxismax) at (4.5,0);
    \coordinate (yaxismin) at (0,-0.5);
    \coordinate (yaxismax) at (0,4.5);
    \draw [ultra thick, black,-latex] (xaxismin) -- (xaxismax); 
    \draw [ultra thick, black,-latex] (yaxismin) -- (yaxismax);
 
     \draw [gray] (0,1) -- (4.5,1);
     \draw [gray] (0,2) -- (4.5,2);
     \draw [gray] (0,3) -- (4.5,3);
     \draw [gray] (0,4) -- (4.5,4);

     \draw [gray] (1,0) -- (1,4.5);
     \draw [gray] (2,0) -- (2,4.5);
     \draw [gray] (3,0) -- (3,4.5);
     \draw [gray] (4,0) -- (4,4.5);

     \draw [blue] (0,0) -- (4,2) -- (2,4) -- (0,0);

     \filldraw [red] (0,0) circle (4pt);
    \filldraw [red] (2,4) circle (4pt);
    \filldraw [red] (4,2) circle (4pt);
    \filldraw [red] (2,2) circle (4pt);

    \draw [] node[below] at (1,0) {$1$}; 
    \draw [] node[left] at (0,1) {$1$}; 

   \endscope
\end{tikzpicture}

\caption{An $\R^2$-circuit $\lambda = (\frac{1}{3},-1,\frac{1}{3},\frac{1}{3})$ of an affine matroid supported on 
 $\cA = \{ (0,0)^T, (2,2)^T, (2,4)^T, (4,2)^T \}$ visualized in terms of its support.}
\label{fi:circuit1}
\end{figure}
\fi

Sublinear circuits appear naturally in the study of non-negative polynomials and,
more generally, of non-negative exponential sums 
$\sum_{\alpha \in \cA} c_{\alpha} \exp(\alpha^T x)$. In the framework of exponential sums,
 Murray, Chandrasekaran and Wierman \cite{mcw-2019} have shown that
 the set of exponential sums (also denoted as \emph{signomials})
 \[
  \sum_{\alpha \in \cA} c_{\alpha} \exp({\alpha^T x}) 
\]
which have at most one negative term and which are non-negative on $X$
can be characterized in terms of a relative entropy program. Sums of such
exponential sums are non-negative as well.
The cone of exponential sums which admit such a non-negativity 
certificate is called the \emph{$X$-SAGE cone} (or \emph{conditional SAGE cone})
supported on $\cA$ and is denoted $C_X(\cA)$. Here, the acronym  
SAGE stands for \emph{Sums of Arithmetic-Geometric Exponentials}
\cite{chandrasekaran-shah-2016}. This cone yields non-negativity
certificates for a subclass of polynomials and signomials and thus provides
a complement to non-negativity certificates based on sums of squares. It
is possible to combine these techniques, see \cite{karaca-2017}.

The introduction and the study of sublinear circuits
is motivated by the following guiding questions:
\begin{enumerate}
\item For $f \in C_X(\mathcal{A})$, there is often more than
  one way to write $f$ as a sum of exponential sums 
  which have only one negative
  term and which are non-negative on $X$. Are there distinguished
  representations among them?
\item Can the $X$-SAGE cone be naturally decomposed as a Minkowski
  sum of smaller subcones?
\item How can the convex geometric properties, such as the extremal
  rays, of the $X$-SAGE cone be characterized?
\end{enumerate}
Here, the second and the third question can be seen as geometric
viewpoints of the first question.
By \cite{mnt-2020},
the conditional SAGE cone $C_X(\mathcal{A})$ can
be decomposed as a Minkowski sum, where each non-trivial
summand refers to the $X$-SAGE exponentials induced by a sublinear circuit (see Proposition~\ref{prop:decomp1} 
for a formal statement). Therefore, the sublinear circuits can
be seen as a convex-combinatorial core underlying the conditional SAGE cone.
In the unconstrained setting, the circuit viewpoint has
been used prominently in the works of Reznick \cite{reznick-1989}, 
Iliman and de Wolff \cite{iliman-dewolff-resmathsci} 
as well as Pantea, Koeppl and Craciun \cite{Pantea2012} on non-negative 
polynomials.

One step further, a reducibility concept for sublinear circuits provides 
a non-redundant decomposition of the conditional SAGE cone
in terms of \emph{reduced circuits}. This reducibility
notion generalizes the reducibility 
notion for the unconstrained situation which
was introduced in \cite{knt-2020}, see also \cite{FdW-2019}.
The reduced $\R^n$-circuits are the key concept to characterize the
extremal rays of the unconstrained SAGE cone, since the reduced
$\R^n$-circuits induce extremal rays.
In generalization of this, the reduced
sublinear circuits facilitate to study the extremal rays of the 
$X$-SAGE cone \cite{mnt-2020},
see Proposition~\ref{prop:decomp2} for a formal statement.

From a more general point of view, 
sublinear circuits generalize the combinatorial concepts
known from an affine matroid, by taking additionally into account
a convex constraint set $X$. 
As such, sublinear circuits enlarge the tool set 
of convex-combinatorial techniques in algebraic geometry and algebraic optimization,
see \cite{bpt-book,dLHK-book,joswig-theobald-book,
michalek-sturmfels-book,sturmfels-gbcp} for general 
background on the rich connections between these disciplines.

In the current paper, we study sublinear circuits for the situation that $X$ is polyhedral.
In this setting, the sublinear circuits can be exactly characterized in terms of the
normal fan of a certain polyhedron, see Proposition~\ref{prop:polyhedron_x_finite_circuits}.
This induces a rich polyhedral-combinatorial structure and makes these sublinear
circuits amenable to effective computations. For polyhedral
$X$, the number of sublinear circuits is finite, and this gives decompositions 
of the $X$-SAGE cones into \emph{finitely} many summands 
referring to the $X$-SAGE exponentials induced by a sublinear circuit.

Among the class of polyhedra, polyhedral cones exhibit particularly nice properties
and were in the focus of attention in earlier treatments.
Note that, as a very particular case, the unconstrained setting $X=\R^n$, which is
treated in \cite{FdW-2019,knt-2020,mcw-newton-poly}, also falls into the 
class of polyhedral cones.
The univariate case $\R_+$ was studied in detail in \cite{mnt-2020}.
Moreover, every univariate case can be transformed to one of the two conic
cases 
$\R$ (unconstrained case), 
$\R_+$ (one-sided infinity interval), or to the non-conic $[-1,1]$ (compact interval).
In the multivariate case, the polyhedra $\R^n$ (unconstrained case),
$\R_+^n$ (non-negative orthant) and the cube $[-1,1]^n$ provide prominent
examples. In contrast to the unconstrained case and to the non-negative orthant,
the cube $[-1,1]^n$ provides a non-conic case.

The goal of the current paper is to develop techniques for handling sublinear circuits,
which also provide an access towards approaching non-conic polyhedral sets.

\subsection*{Contributions}
1. We reveal some precise connections between sublinear circuits and their
  supports, see Lemma~\ref{le:support-lemma1}.
  In particular, we show that in general sublinear circuits are not uniquely
  determined by their supports, see Example~\ref{ex:counterex-support}.

2. We develop necessary and sufficient conditions for identifying $X$-circuits
  based on support conditions. See Theorems~\ref{th:convex_hull_lambda}
  and \ref{th:xcircuit_nonnegative_orthant}.

3. We give conditions for identifying reduced $X$-circuits which generalize
    the known characterizations for the unconstrained case. See 
    Theorems~\ref{th:combinatorial-reducibility1} 
    	and~\ref{th:connect-reduced-to-rn-1}.
  
4. Building upon the criteria for sublinear circuits, we
  study the prominent cases of the non-negative orthant $\R_+^n$ and 
  the cube $[-1,1]^n$ in detail,
  in particular the planar case and with regard to small support sets.
  Specifically, for $\mathcal{A} = \{ (i,j) \, : \, 1 \le i,j \le 3\} \subset \R^2$ and $X=[-1,1]^2$, there 
  are 132 circuits and 24 reduced $X$-circuits, which we classify.
  See Sections~\ref{se:cube} and~\ref{se:reducibility}.

5. As a specific consequence, we can exactly determine the extreme rays of 
  the univariate $[-1,1]$-SAGE cone, 
  see Theorem~\ref{extreme_ray_one_dim_polytope}.
\smallskip

For further recent work on the techniques for certifying non-negativity of
signomials and polynomials based on the SAGE cone and its variants, 
see \cite{averkov-2019,DKdW-2021,NaumannTheobald2021,
  wjyp-2020,wang-magron-2020}.
  
The paper is structured as follows. After collecting relevant concepts of
sublinear circuits and non-negative signomials in Section~\ref{se:prelim},
we study the connection of $X$-circuits and their supports in Section~\ref{se:supports}.
Section~\ref{se:conditions} deals with necessary and sufficient conditions
for sublinear circuits, and Section~\ref{se:cube} focuses on the case
of the cube $[-1,1]^n$. In Section~\ref{se:reducibility}, we provide criteria
for reduced sublinear circuits, which gives as a consequence the
characterization of the extreme rays of the $[-1,1]$-SAGE cone.
Section~\ref{se:conclusion} concludes the paper.

\subsection*{Acknowledgment.}
We thank Riley Murray for valuable discussions 
and the anonymous referees for their helpful
suggestions.
 
\section{Preliminaries\label{se:prelim}}
 
Throughout the paper, the symbol $\mathbf{0}$ denotes the zero vector,
$\oneb$ denotes the all-ones vector and $[m]$ abbreviates the set $\{1,\ldots, m\}$ for $m\in\N$.
For a given convex subset $X \subset \R^n$, denote by 
$\sigma_X(y) = \sup\{ y^T x : x \in X \}$ its support function.

\subsection{\texorpdfstring{$X$}{X}-circuits}For a non-empty
convex set $X$ and finite $\mathcal{A} \subset \R^n$,
we consider $X$-circuits as defined in the Introduction, where we note 
that the two defining conditions can also be expressed in terms of the
support function: then (1) becomes the condition 
$\sigma_X(-\cA \nu^\star) < \infty$ and
in (2), the mapping $\nu \mapsto \sigma_X(-\cA \nu)$ occurs.

A sublinear circuit $\lambda \in N_{\beta}$
is called \emph{normalized} if $\lambda_{\beta} = -1$, in which case the condition
${\bf 1}^T \lambda = 0$ in the definition of $N_{\beta}$ implies
$\sum_{\alpha \neq \beta} \lambda_{\alpha} =1$. For normalized $X$-circuits, we usually
employ the symbol $\lambda$, whereas we use the symbol $\nu$ for $X$-circuits which
are not necessarily normalized. For $X\subset\R^n$, denote by
$\Lambda_X(\cA, \beta)$ the set of normalized $X$-circuits of $\cA$ with negative entry corresponding to $\beta\in\cA$ and by $\Lambda_X(\cA):=\bigcup_{\beta \in \cA} \Lambda_X(\cA,\beta)$ the set of \textit{all} normalized $X$-circuits of $\cA$.
For example, in the univariate case $X=\R$ with 
$\mathcal{A} = \{\alpha_1, \ldots, \alpha_m \} \subset \R$,
the normalized $X$-circuits are
    \begin{align}
    \lambda \ = \ 
    \frac{\alpha_k-\alpha_j}{\alpha_k - \alpha_i} e^{(i)} 
    - e^{(j)} 
    + \frac{\alpha_j - \alpha_i}{\alpha_k - \alpha_i} e^{(k)}
    \quad \text{ for } i < j < k,\label{eq:Rn_circuits}
    \end{align}
where $e^{(i)}$ denotes the $i$-th unit vector in $\R^m$.
It is possible that a given support set $\cA \subset \R^n$ 
has no $\R^n$-circuits, but then every 
$\alpha \in \cA$ is an extreme point of $\conv\cA$.

\begin{example}\label{ex:3-point-support}
In the context of the conditional SAGE cone, we
can assume without loss of generality, that the convex set $X$ is closed.
In the one-dimensional case, up to translation and additive inversion, 
each closed, convex set is of the form
$X^{(1)}=\R$, 
$X^{(2)}=\R_+$ or $X^{(3)}=[-1,1]$.
For the support set $\cA=\{0,1,2\}$, it is instructive to list the
sublinear circuits with respect to the three sets $X^{(1)}$,
$X^{(2)}$ and $X^{(3)}$. The set $\Lambda^{(1)}$ of $X^{(1)}$-circuits is
$\R_+(1,-2,1)^T$, which is a special case of~\eqref{eq:Rn_circuits}.
The set $\Lambda^{(2)}$ of $X^{(2)}$-circuits is 
\[
  \Lambda^{(2)} = \Lambda^{(1)} \cup
   \R_+(0,-1,1)^T \cup \R_+(-1,0,1)^T \cup \R_{+}(-1,1,0)^T,
\]
and this is a special case of Proposition~\ref{prop:univariate1} below.
In particular, the element $(-1,1,0)^T$ is not an $X^{(1)}$-circuit,
because
\[
  \sigma_{X^{(1)}}(-\cA (-1,1,0)^T) = \sigma_{X^{(1)}}(-1)
  = \sup_{x \in \R}(-x) = \infty.
\]
The set $\Lambda^{(3)}$ of $X^{(3)}$-circuits is 
\[
  \Lambda^{(3)} = \Lambda^{(2)} \cup
   \R_+(0,1,-1)^T \cup \R_+(1,0,-1)^T \cup \R_{+}(1,-1,0)^T,
\]
which is a special case of Proposition~\ref{prop:circuits-univariate2}
proven in Section~\ref{se:supports}.
Note that, for example, the element $(1,-1,0)^T$ is not an $X^{(2)}$-circuit, 
as 
\[
  \sigma_{X^{(2)}}(-\cA (1,-1,0)^T) = \sigma_{X^{(2)}}(1)=\sup_{x\ge 0} x=\infty.
\]
\end{example}	

For polyhedral $X$, the sublinear circuits
can be characterized in terms of normal fans of polyhedra. We refer the
reader for background on normal fans to \cite[Chapter~7]{Ziegler}
(for the bounded case of polytopes), \cite[Section~5.4]{gkz-1994} or
\cite[Chapter~2]{sturmfels-gbcp}.
For each face $F$ of a polyhedron $P$, let 
\[
    \Nsf_P(F) = \{ w \,:\, z^T w = \sigma_P(w) ~\forall \, z\in F\}
\]
be the associated \emph{outer normal cone}.

The support function $\sigma_P$
of a polyhedron $P$ is linear on every outer
normal cone,
and the linear representation may be given by $\sigma_P(w) = z^T w$
for any $z \in F$.
The \textit{outer normal fan} of $P$ is the collection of all outer normal cones,
\[
    \mathcal{O}(P) = \{ \Nsf_P(F) \,:\, F \text{ is a face of } P \}.
\]
For a convex cone $K \subset \R^n$, denote by
$K^* := \{c\in \R^n: c^Tx\ge 0 \text{ for all }x\in K\}$ the \emph{dual cone}
and by $K^{\circ} := -K^*$ the \emph{polar}. For a set $S \subset \R^n$,
let 
$\rec(S) := \{ t \,:\, \exists s \in S \text{ such that } s + \lambda t \in S \; \forall\; \lambda \geq 0 \}$ denote its recession cone. Using these notations, the support of
$\mathcal{O}(P)$ coincides with $\rec(P)^{\circ}$.
The full-dimensional linearity domains of the support function $\sigma_P$
are the outer normal cones
of the vertices of $P$ (see also \cite[Section~1]{fillastre-izmestiev-2017}).

\begin{proposition}\cite{mnt-2020}
\label{prop:polyhedron_x_finite_circuits}
Let $X$ be a polyhedron. Then $\nu \in N_\beta \setminus \{\zerob\}$ is an
    $X$-circuit if and only if
     $\cone\{\nu\}$ is a ray in $\mathcal{O}(-\cA^T X + N_{\beta}^\circ)$.
    As a consequence, there are only finitely many normalized $X$-circuits.
\end{proposition}

If $X$ is a polyhedral cone, the situation simplifies, because the support 
function $\sigma_X(-\mathcal{A} \nu)$ of a circuit $\nu$ can
only attain the values zero and infinity. Namely, since 
$\mathcal{O}(-\cA^T X + N_{\beta}^\circ) =  (\cA^T X)^* \cap N_{\beta}$
and
\[
  (\cA^T X)^* 
=  \{\nu \, : \, \nu^T y \ge 0 \; \forall y \in \cA^T X\}
= \{\nu \, : \, (\cA\nu)^T x \ge 0 \; \forall x \in X\} 
= \{\nu \, : \, \sigma_X(-\cA\nu) \le 0\},
\]
the $X$-circuits $\nu \in N_{\beta}$ 
are precisely the edge generators of the polyhedral cone
$\{\nu \in N_{\beta} \, : \, \sigma_X(-\cA \nu) \le 0\}$.

In the univariate case with $\mathcal{A} = \{\alpha_1, \ldots, \alpha_m \} \subset \R$,
the sublinear circuits for the univariate cone $[0,\infty)$ have been determined in \cite{mnt-2020}:

\begin{proposition}\label{prop:univariate1} 
For $X = [0,\infty)$ and $\mathcal{A} = \{\alpha_1, \ldots, \alpha_m \} \subset \R$
with $\alpha_1 < \cdots < \alpha_m$,
the normalized $X$-circuits $\lambda \in \R^m$ are the vectors  either of the form
$\lambda = e^{(k)} - e^{(j)}$ for $j < k$ or of the form
    \begin{equation}
    \label{eq:circuits-r1}
    \lambda = 
       \frac{\alpha_k-\alpha_j}{\alpha_k - \alpha_i} e^{(i)} 
      - e^{(j)} 
       + \frac{\alpha_j - \alpha_i}{\alpha_k - \alpha_i} e^{(k)}
      \quad\text{ for } i < j < k.
    \end{equation}
\end{proposition}

Note that the $X$-circuits of the second form are exactly the $\R$-circuits
from~\eqref{eq:Rn_circuits}.

\begin{remark}\label{rem:X_circuits_as_facets}
By Proposition~\ref{prop:polyhedron_x_finite_circuits}, the
$X$-circuits of $\cA$ are the outer normal vectors to facets of polyhedra $P = -\cA^T X + N_{\beta}^\circ$ (for some $\beta$). 
As $N_{\beta}$ is pointed, $P$ is always full-dimensional.
\end{remark}

\begin{example}
If $X$ is a convex cone, then the second condition in the definition of $X$-circuits
simplifies, because the support function evaluates to $0$ whenever 
it is finite. Consider the conic sets $X^{(1)}=\R^2$ and $X^{(2)}=\R_+^2$
with respect to the support set $\cA=\{(0,0)^T,(0,4)^T,(4,0)^T,(1,1)^T\}$, 
as illustrated in Figure~\ref{fi:circuit2dim}.
Three points of $\cA$ are vertices of the convex hull of $\cA$,
and the point $(1,1)^T$ is contained in the relative interior of the convex hull of $\cA$. 
The set $\Lambda^{(1)}$ of $X^{(1)}$-circuits is 
$\R_+(2,1,1,-4)^T$, and the set $\Lambda^{(2)}$ of $X^{(2)}$-circuits
is
\[
  \Lambda^{(2)} = \Lambda^{(1)} \cup 
  \R_+(0,3,1,-4)^T \cup \R_+(0,1,3,-4)^T.
\]
Similar to the arguments for the one-dimensional-cases
in Example~\ref{ex:3-point-support},
the $X^{(2)}$-circuit $(0,1,3,-4)^T$ is not an $X^{(1)}$-circuit,
as the resulting support function with respect to $X^{(1)}$
is not finite anymore: 
$\sigma_{\R^2}(-\cA(0,1,3,-4))=\sup_{x_1\in\R}-8x_1=\infty$.
\end{example}

\begin{figure}[ht]
        \begin{minipage}{8cm}
	\begin{tikzpicture}
	\useasboundingbox (-4.5,-4.5) rectangle (5,4);
	\scope[transform canvas={scale=.7}]
	\coordinate (Origin)   at (0,0);
	\coordinate (xaxismin) at (-4.5,0);
	\coordinate (xaxismax) at (4.5,0);
	\coordinate (yaxismin) at (0,-4.5);
	\coordinate (yaxismax) at (0,4.5);
	\draw [ultra thick, black,-latex] (xaxismin) -- (xaxismax); 
	\draw [ultra thick, black,-latex] (yaxismin) -- (yaxismax);
	
	\draw [purple] (3.5,0.5) -- (0.5,3.5) ;
	\draw [purple] (3.5,1.5) -- (1.5,3.5) ;
	\draw [purple] (3.5,2.5) -- (2.5,3.5) ;
	\draw [purple] (3.5,-.5) -- (-.5,3.5) ;
	\draw [purple] (3.5,-1.5) -- (-1.5,3.5) ;
	\draw [purple] (3.5,-2.5) -- (-2.5,3.5) ;
	\draw [purple] (3.5,-3.5) -- (-3.5,3.5) ;
	\draw [purple] (2.5,-3.5) -- (-3.5,2.5) ;
	\draw [purple] (1.5,-3.5) -- (-3.5,1.5) ;
	\draw [purple] (.5,-3.5) -- (-3.5,0.5) ;
	\draw [purple] (-.5,-3.5) -- (-3.5,-.5) ;
	\draw [purple] (-1.5,-3.5) -- (-3.5,-1.5) ;
	\draw [purple] (-2.5,-3.5) -- (-3.5,-2.5) ;
	
	\draw [cyan] (0.25,0.25) -- (3.75,3.75) ;
	\draw [cyan] (1.25,0.25) -- (3.75,2.75) ;
	\draw [cyan] (2.25,0.25) -- (3.75,1.75) ;
	\draw [cyan] (3.25,0.25) -- (3.75,0.75) ;
	\draw [cyan] (0.25,1.25) -- (2.75,3.75) ;
	\draw [cyan] (0.25,2.25) -- (1.75,3.75) ;
        \draw [cyan] (3.25,0.25) -- (3.75,0.75) ;
        \draw [cyan] (0.25,3.25) -- (0.75,3.75) ;

	\draw [] node[below] at (1,0) {$1$}; 
	\draw [] node[left] at (0,1) {$1$}; 
	
	\endscope
	\end{tikzpicture}
	\end{minipage}
	\begin{minipage}{8cm}
	\begin{tikzpicture}
	\useasboundingbox (-4.5,-4.5) rectangle (5,4);
	\scope[transform canvas={scale=.7}]
	\coordinate (Origin)   at (0,0);
	\coordinate (xaxismin) at (-4.5,0);
	\coordinate (xaxismax) at (4.5,0);
	\coordinate (yaxismin) at (0,-4.5);
	\coordinate (yaxismax) at (0,4.5);
	\draw [ultra thick, black,-latex] (xaxismin) -- (xaxismax); 
	\draw [ultra thick, black,-latex] (yaxismin) -- (yaxismax);
	
	\draw [gray] (-4.5,1) -- (4.5,1);
	\draw [gray] (-4.5,2) -- (4.5,2);
	\draw [gray] (-4.5,3) -- (4.5,3);
	\draw [gray] (-4.5,4) -- (4.5,4);
	\draw [gray] (-4.5,-1) -- (4.5,-1);
	\draw [gray] (-4.5,-2) -- (4.5,-2);
	\draw [gray] (-4.5,-3) -- (4.5,-3);
	\draw [gray] (-4.5,-4) -- (4.5,-4);
	
	\draw [gray] (1,-4.5) -- (1,4.5);
	\draw [gray] (2,-4.5) -- (2,4.5);
	\draw [gray] (3,-4.5) -- (3,4.5);
	\draw [gray] (4,-4.5) -- (4,4.5);
	\draw [gray] (-1,-4.5) -- (-1,4.5);
	\draw [gray] (-2,-4.5) -- (-2,4.5);
	\draw [gray] (-3,-4.5) -- (-3,4.5);
	\draw [gray] (-4,-4.5) -- (-4,4.5);
		
	\filldraw [red] (0,0) circle (4pt);
	\filldraw [red] (0,4) circle (4pt);
	\filldraw [red] (4,0) circle (4pt);
	\filldraw [red] (1,1) circle (4pt);
	
	\draw [] node[below] at (1,0) {$1$}; 
	\draw [] node[left] at (0,1) {$1$}; 
	
	\endscope
	\end{tikzpicture}
	\end{minipage}
	\vspace*{-6ex}
	
	\caption{The sets \textcolor{purple}{$X^{(1)}=\R^2$} and \textcolor{cyan}{$X^{(2)}=\R_+^2$} and the support set \textcolor{red}{$\cA=\{(0,0)^T,(0,4)^T,(4,0)^T,(1,1)^T\}$}}
	\label{fi:circuit2dim}
\end{figure}
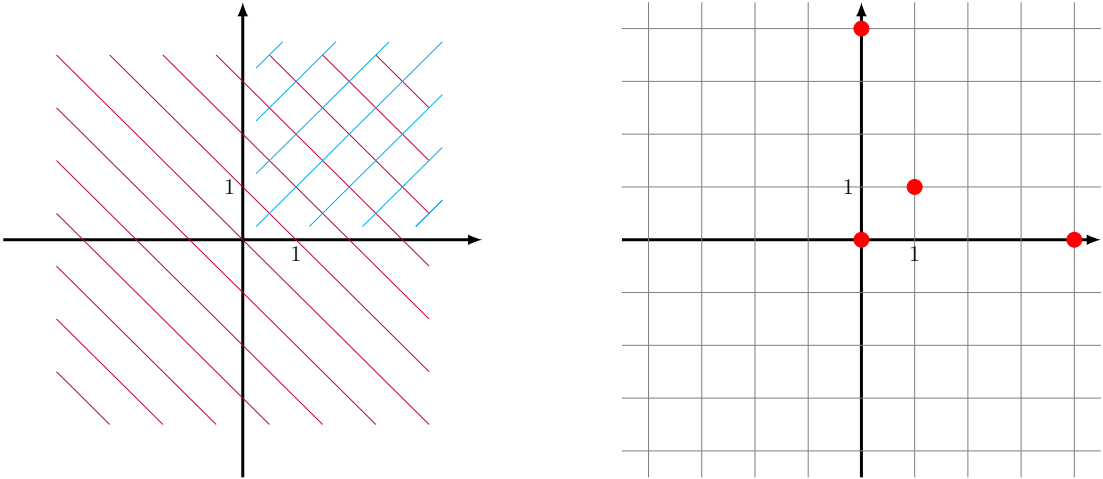

\subsection{Non-negativity of signomials}

We consider the cone $C_X(\cA)$ of $X$-SAGE signomials supported on $\cA$
\cite{mcw-2019}, 
which was informally introduced in the Introduction.
For $\beta \in \cA$, set
\begin{equation}\label{eq:def_c_age}
    C_X(\cA,\beta) = \left\{ f \,:\, f =\sum_{\alpha\in\cA}c_{\alpha}\exp(\alpha^T x) \text{ is non-negative on $X$},~ c_{\setminus \beta} \geq \zerob \right\},
\end{equation}
called the \emph{$X$-AGE cone supported on $\cA$ with respect to $\beta$.}
By \cite{mcw-2019},
$C_X(\cA)$ decomposes as $C_X(\cA) = \sum_{\beta \in \cA} C_X(\cA,\beta)$.

Given a vector $\lambda \in N_{\beta}$ with $\lambda_{\beta} = -1$, the 
\emph{$\lambda$-witnessed AGE cone} $C_X(\cA,\lambda)$ is defined as
\begin{equation}\label{eq:weighted_prim_x_age_powercone}
    C_X(\cA,\lambda) = \left\{~ \sum_{\alpha\in\cA}c_{\alpha}\exp(\alpha^T x) \,:\, \prod_{\alpha \in \lambda^+} \left(\frac{c_\alpha}{\lambda_\alpha} \right)^{\lambda_\alpha} \geq - c_\beta \exp\left(\sigma_X(-\cA \lambda)\right) ,~ c_{\setminus \beta} \geq \zerob \right\}.
\end{equation}
All signomials in $C_X(\cA,\lambda)$ are non-negative over $X$. Moreover, for polyhedral 
$X$, the conditional SAGE cone can be naturally decomposed into the Minkowski sum of a finite 
set of $\lambda$-witnessed cones, where $\lambda$ runs over
the normalized $X$-circuits.

\begin{proposition}\cite{mnt-2020}
  \label{prop:decomp1}
    Let $X \subset \R^n$ be a polyhedron and $\Lambda_X(\cA)$ be nonempty.
    Then the conditional SAGE cone $C_X(\cA)$ decomposes as the finite Minkowski sum
    \begin{equation}
        C_X(\cA) = \sum_{\lambda \in \Lambda_X(\cA)}  C_X(\cA,\lambda) .
    \end{equation}
\end{proposition}

\subsection{Reduced circuits and non-negativity of signomials\label{se:prelim-reduced}}

In general, the representation in Proposition~\ref{prop:decomp1} can include redundancies. Using a reducibility concept of circuits, which takes into account
the value $\sigma_X(-\cA \nu)$ of an $X$-circuit $\nu$, an irredundant representation
can be given.
The \emph{extended form} of an $X$-circuit $\nu \in \R^{\cA}$ is defined as
$(\nu,\sigma_X(-\cA \nu)) \in \R^{\mathcal{A}} \times \R$. The set
\[
  G_X(\mathcal{A}) \ = \ \cone (\{ (\nu,\sigma_X(-\cA\nu)) \ : \ \lambda \in 
    \Lambda_X(\cA) \} \cup \{(\mathbf{0}, 1)\})
\]
is called the \emph{circuit graph} of $(\mathcal{A},X)$.
Whenever we consider the circuit graph or the reduced sublinear circuits defined subsequently, 
we will tacitly assume that the functions $x \mapsto \exp(\alpha^T x)$, 
$\alpha \in \cA$, are linearly independent on $X$. Then the set $G_X(\mathcal{A})$ is 
pointed and closed (see \cite{mnt-2020}). 

\begin{definition}\label{def:reduced_circuits}
  An $X$-circuit $\nu$ is called \emph{reduced} if its extended form
  generates an extreme ray of $G_X(\cA)$.
  Denote by $\Lambda_X^\star(\cA)$ the set of normalized reduced $X$-circuits.
\end{definition}

For polyhedral $X$, the conditional SAGE cone $C_X(\cA)$
can be decomposed into the Minkowski sum
of a finite set of $\lambda$-witnessed cones, where $\lambda$ runs over
the reduced $X$-circuits.

\begin{proposition}\cite{mnt-2020}
  \label{prop:decomp2}
    Let $X \subset \R^n$ be a polyhedron and $\Lambda_X(\cA)$ be non-empty.
    Then the conditional SAGE cone $C_X(\cA)$ decomposes as the finite Minkowski sum
    \begin{equation}\label{eq:decomp2}
        C_X(\cA) = \sum_{\lambda \in \Lambda^\star_X(\cA)}  C_X(\cA,\lambda)  .
    \end{equation}
    Moreover, there does not exist a proper subset $\Lambda \subsetneq \Lambda_X^\star(\cA)$ with 
    $$C_X(\cA) = \sum_{\lambda \in \Lambda} C_X(\cA,\lambda).$$
\end{proposition}

In the univariate case with $\cA = \{\alpha_1, \ldots, \alpha_m\}$ sorted ascendingly,
we have
\begin{eqnarray*}
    \Lambda^\star_{\R}(\cA)
    & = & \left\{
       \left(\frac{\alpha_{i+1}-\alpha_i}{\alpha_{i+1} - \alpha_{i-1}}\right)e^{(i-1)} 
      - e^{(i)} 
       + \left(\frac{\alpha_i - \alpha_{i-1}}{\alpha_{i+1} - \alpha_{i-1}}\right)e^{(i+1)}
   \ : \  2 \le i \le m-1 \right\} \\
    \text{ and } \Lambda^\star_{[0,\infty)}(\cA)
    & = & \Lambda^\star_{\R}(\cA) \cup \{e^{(2)} - e^{(1)}\}.
\end{eqnarray*}
See \cite{FdW-2019} for $\Lambda^\star_{\R}(\cA)$ and
\cite{mnt-2020} for $\Lambda^\star_{[0,\infty)}(\cA)$.
  
\section{\texorpdfstring{$X$}{X}-circuits and their supports\label{se:supports}}

In this section, we study the relationship between $X$-circuits and their supports. 
We begin with a study of the compact univariate case $[-1,1]$. This complements
the known cases $\R$ from the Introduction and $[0,\infty)$ from Proposition~\ref{prop:univariate1}.

\begin{proposition}
	\label{prop:circuits-univariate2}
	Let $X = [-1,1]$ and
	$\mathcal{A} = \{\alpha_1, \ldots, \alpha_m \} \subset \R$
	with $\alpha_1 < \cdots < \alpha_m$. An element $\lambda \in\bigcup_{\beta\in\cA}N_\beta$ is a normalized $X$-circuit if and only if it is of the following form: 
	\begin{enumerate}
		\item $\lambda = e^{(j)} - e^{(i)}$ for $i \ne j$, or \smallskip

		\item $\displaystyle{\lambda = 
		\frac{\alpha_k-\alpha_j}{\alpha_k-\alpha_i}
		e^{(i)} -e^{(j)} + 
		\frac{\alpha_j-\alpha_i}{\alpha_k-\alpha_i}
		e^{(k)}}$
	    for $i <j < k $.
	\end{enumerate}
\end{proposition}

	\begin{proof}
	Fix $j\in [n]$ and write $N_j := N_{(\alpha_j)}$ for short.
	By Proposition~\ref{prop:polyhedron_x_finite_circuits},
	the $X$-circuits are the vectors spanning the rays in the outer normal cone of
	the polyhedron
	\begin{align*}
	P \ & = \  -\cA^TX+N_{j}^\circ\\
	\ & = \ \conv\{(\alpha_1, \ldots, \alpha_m)^T,
	-(\alpha_1, \ldots, \alpha_m)^T \}
	+\R\cdot \mathbf{1}-\sum\limits_{i\ne j}\pos e^{(i)}\\
	\ & = \ \{\theta (\alpha_1, \ldots, \alpha_m)^T+\mu\mathbf{1}: -1\le \theta\le 1, \mu\in\R\}-\sum\limits_{i\ne j}\pos e^{(i)}.
	\end{align*}
	Hence, a point $w$ is contained in $P$ if and only if
	\begin{align*}
	w_i \ \le
	& \quad	\theta\alpha_i+\mu\text{ for }i\ne j \; \text{ and } \;
	w_j \ = \theta\alpha_j+\mu
	\quad \text{ for } \theta\in[-1,1] 
	\text{ and } \mu\in\R.
	\end{align*}
	By eliminating $\mu$, this is equivalent to
	\begin{align*}
	w_j-w_i+\theta (\alpha_i-\alpha_j)\ge 0 \; \text{ for all }
  i\in [m] \setminus \{j\}, \; -1\le \theta \le 1.
	\end{align*}
	Eliminating $\theta$ then gives
	\begin{align*}
&	\frac{w_j-w_i}{\alpha_j-\alpha_i}\begin{cases}
\quad \le \theta \le 1 & \text{ if } \alpha_i >\alpha_j,\\
\quad \ge \theta \ge -1 & \text{ if } \alpha_i <\alpha_j,
\end{cases} 
\end{align*}
which yields
$\displaystyle \frac{w_j-w_i}{\alpha_j-\alpha_i} \ge \frac{w_k-w_j}{\alpha_k-\alpha_j} \text{ for all }i,k\in[m] \text{ with } i< j < k$ 
and $w_i-w_j \le | \alpha_i - \alpha_j |$ for all $i \in [m] \setminus \{j\}$.
Hence,
		\begin{align}
		\label{eq:defp1}
P \ = \ & \left\{w\in\R^m:	\; 	w_i-w_j \le |\alpha_i-\alpha_j|
  \text{ for } i\in [m] \setminus \{j\} \text{ and } \right.\\
      \label{eq:defp2}
	 &  \left. 
	  w_i(\alpha_k-\alpha_j)
	 -w_j(\alpha_k-\alpha_i)
	 +w_k(\alpha_j-\alpha_i) 
	 \le 0 \text{ for }i,k\in[m] \text{ with } i < j < k \right\}.
	\end{align}
We claim that none of the inequalities in the definition of~$P$ is redundant. Namely, for each inequality 
\begin{align*}
	w_i(\alpha_k-\alpha_j)
	-w_j(\alpha_k-\alpha_i)
	+w_k(\alpha_j-\alpha_i) 
	\le 0
\end{align*}
in~\eqref{eq:defp2},
the point $e^{(i)} + e^{(j)} + e^{(k)}$ satisfies this particular
inequality with equality and all of the other inequalities strictly. Similarly,
for the inequalities 
in~\eqref{eq:defp1},
it suffices to
consider the point $\alpha_j e^{(i)} + \alpha_i e^{(j)}$ in case $i < j$ 
and $\alpha_i e^{(i)} + \alpha_j e^{(j)}$ in case $i > j$. By 
Remark~\ref{rem:X_circuits_as_facets}, the polyhedron $P$ is full-dimensional.

	Hence, by Proposition~\ref{prop:polyhedron_x_finite_circuits},
	the normalized $X$-circuits in $N_j$ are exactly the ones
	given in the statement of the theorem.
\end{proof}

\subsection*{The supports of \texorpdfstring{$X$}{X}-circuits.}

As stated in the Introduction,
in the classical case of affine circuits, the normalized
circuits are uniquely determined by
their supports. Moreover, 
as a consequence of Theorem~\ref{prop:circuits-univariate2}, in
the case $X=[-1,1]$, the normalized $X$-circuits are uniquely determined by their signed
supports. As explained in the following, this phenomenon does not extend to
sublinear circuits for arbitrary sets.

In the case of sublinear circuits supported on two elements, the two non-zero 
entries are additive inverses of each other, so that, for a given $\beta$
and a given support, indeed this signed support uniquely determines the
circuit up to a positive factor.
In order to exhibit the mentioned phenomenon, we present a
counterexample with support size~3.

\begin{example}
\label{ex:counterex-support}
Let $\mathcal{A} = \{ \alpha_1, \alpha_2, \alpha_3 \} =
  \left\{ 
  (0,0)^T, (1,0)^T, (0,1)^T
  \right\} \subset \R^2$. We show that for $\beta := \alpha_1$, there 
are two non-proportional circuits which are supported on all three elements
of $\mathcal{A}$. Specifically, we construct an example, in which
\[
  \nu^{(1)} := (-2,1,1)^T \quad \text{ and } \quad \nu^{(2)} := (-3,1,2)^T
\]
are sublinear circuits. Note that both of them have the same signed support,
but they are not multiples of each other.
Observe that
\[ 
  -\mathcal{A} \nu^{(1)} = (-1,-1)^T,
  \qquad
   -\mathcal{A} \nu^{(2)} = (-1,-2)^T.
\]
We set up $X$ in such a way that $(-1,-1)^T$ and $(-1,-2)^T$
are normal vectors of $X$. For example, choose $X$ as the cone in $\R^2$
spanned by $(-1,1)^T$ and $(2,-1)^T$.
We obtain
\[
  - \mathcal{A}^T X \ = \ \pos \left\{
    - \left( \begin{array}{cc}
      0 & 0 \\
      1 & 0 \\
      0 & 1
      \end{array} \right)
      \left( \begin{array}{r} -1 \\ 1 \end{array} \right),
      - \left( \begin{array}{cc}
      0 & 0 \\
      1 & 0 \\
      0 & 1 \\
      \end{array} \right)
      \left( \begin{array}{r} 2 \\ -1 \end{array} \right)
      \right\}
      \ = \ \pos  \left\{
      \left( \begin{array}{r} 0 \\ 1 \\ -1 \end{array} \right),
      \left( \begin{array}{r} 0 \\ -2 \\ 1 \end{array}  \right)
      \right\}.  
\]
Since $N_\beta^{\circ} = N_ {(0,0)}^{\circ} = 
\R \cdot (1,1,1)^T + \R \times \R_{\le 0} \times \R_{\le 0}$,
it can be verified (for example, using a computer calculation)
that $\nu^{(1)}$ and $\nu^{(2)}$ are indeed sublinear circuits, and
they are the only ones having a negative component $\nu_\beta$ 
up to scaling by a positive factor.
\end{example}

In the example, the two distinct sublinear circuits 
$\nu^{(i)}$, $1 \le i \le 2$, with identical signed supports, have different 
expressions 
$\mathcal{A} \nu^{(i)}$, that is, 
$\mathcal{A} \nu^{(1)} \neq \mathcal{A} \nu^{(2)}$.
By the following statement,
it is not possible to have two distinct sublinear circuits
with the same signed support and identical non-zero 
values of $\mathcal{A} \nu^{(i)}$.

\begin{lemma}
\label{le:support-lemma1}
Let $\nu^{(1)}$ and $\nu^{(2)}$ be sublinear circuits with the same signed support 
and such that $\mathcal{A} \nu^{(1)} = \mathcal{A} \nu^{(2)}$. Then
$\nu^{(1)}$ and $\nu^{(2)}$ are proportional, and in case
$\mathcal{A} \nu^{(1)} = \mathcal{A} \nu^{(2)} \neq 0$, the equality
$\nu^{(1)} = \nu^{(2)}$ holds.
\end{lemma}

\begin{proof}
Let $\nu^{(1)}$ and $\nu^{(2)}$ have the same signed support with 
$\mathcal{A} \nu^{(1)} = \mathcal{A} \nu^{(2)}$. Set
$\beta$ as the index of the negative component of 
$\nu^{(1)}$ and $\nu^{(2)}$.

Assuming $\nu^{(1)} \neq \nu^{(2)}$, the precondition 
$\supp \nu^{(1)} = \supp \nu^{(2)}$ implies that for
sufficiently small $\varepsilon > 0$, the vectors
\[
  \nu' := \nu^{(1)} - \varepsilon \nu^{(2)} 
  \quad \text{ and }
  \nu'' := \nu^{(1)} + \varepsilon \nu^{(2)}
\]
are contained in $N_{\beta} \setminus \{\mathbf{0}\}$ as well. Observe that
$\sigma_X(-\mathcal{A} \nu') =\sigma_X(-\mathcal{A} \nu^{(1)}) - \varepsilon \sigma_X(-\mathcal{A} \nu^{(2)}) < \infty$ and 
$\sigma_X(-\mathcal{A} \nu'') = 
\sigma_X(-\mathcal{A} \nu^{(1)}) + \varepsilon \sigma_X(-\mathcal{A} \nu^{(2)}) < \infty$.  Moreover,
$\nu^{(1)}$ is a convex combination 
$\nu^{(1)} = \frac{1}{2} \nu' + \frac{1}{2} \nu''$ for which
$\nu \mapsto \sigma_X(-\cA \nu)$ is linear on $[\nu', \nu'']$.

Since 
$\cA\nu^{(1)} = \cA \nu^{(2)}$, the vectors $\nu^{(1)}$ and
$\nu^{(2)}$ are not proportional or we have 
$\mathcal{A} \nu^{(1)} = \mathcal{A} \nu^{(2)} = \mathbf{0}$.
In both cases, if $\nu^{(1)}$ and $\nu^{(2)}$ are non-proportional,
then this contradicts that $\nu^{(1)}$ is a sublinear circuit.
\end{proof}

\section{Necessary and sufficient conditions\label{se:conditions}}

In this section, we obtain some criteria for elements $\nu\in\bigcup_{\beta\in\cA}N_{\beta}$ to be $X$-circuits of some fixed set $X$. These criteria 
only involve the supports rather than the exact values of the coefficients.

For an $X$-circuit $\nu$, let $\nu^+:=\{\alpha:\nu_\alpha\ge 0\}$ and 
$\nu^-$ denote the single index $\beta$
with $\nu_{\beta} < 0$.
First recall that 
in the classical case of affine matroids, any simplicial circuit $\nu$ 
supported on at least three elements has no other support point
except $\nu^-$ contained in the relative interior of the convex hull of all its support points, and the coefficients of $\nu^+$ are positive multiples of the
barycentric coordinates of $\beta$, 
i.e., 
$\relint \conv(\supp \nu) \cap \nu^+ = \emptyset$ and 
$\cA \nu = \mathbf{0}$
(see, e.g., \cite{FdW-2019}).
In the following theorem, we give a generalization of this property to the
case of $X$-circuits.

\begin{theorem}\label{th:convex_hull_lambda}
	Let $\lambda \in  \Lambda_X(\cA,\beta)$ for some $\beta\in\cA$. Then  $\relint\conv(\supp \lambda)\cap\lambda^+=\emptyset$. Moreover, if
	  $\beta\in\conv(\lambda^+)$, then $\cA\lambda= \mathbf{0}$.
\end{theorem}

\begin{proof}
	For the first statement, suppose there exists $\bar{\alpha}\in\lambda^+$ such that 
	$\bar{\alpha}\in\relint\conv(\supp \lambda)$. 
Hence,
	there exist
	$\theta_\alpha\in [0,1)$ 
	for $\alpha\in (\lambda^+ \setminus\{\bar{\alpha}\}) \cup \{\beta\}$
	such that
\[
	\sum\limits_{\alpha\in\lambda^+\setminus\{\bar{\alpha}\}} \theta_\alpha+\theta_\beta=1
	\quad \text{ and } \quad 
	\sum\limits_{\alpha\in\lambda^+\setminus\{\bar{\alpha}\}} \theta_\alpha \alpha+\theta_\beta\beta=\bar{\alpha}.
\]
 Let $\tau\in (0,1]$ be maximal such that  $\tau \theta_\alpha\lambda_{\bar{\alpha}}\le \lambda_{\alpha}$ for 
 $\alpha\in (\lambda^+\setminus\{\bar{\alpha}\}) \cup \{\beta\}$ and $(1+\tau)\lambda_{\bar{\alpha}}<1$. As $\lambda_{\bar{\alpha}}<1$, 
	this does indeed exist. 
	The two vectors $\nu^{(1)}$ and $\nu^{(2)}$ defined by
	\begin{align*}
	\nu^{(1)}_\alpha=  & \begin{cases}
	\lambda_\alpha+\tau\theta_\alpha\lambda_{\bar{\alpha}} & 
	\text{ for }\alpha \in (\lambda^+ \setminus \{\bar{\alpha}\}) \cup \{\beta\}, \\
	(1-\tau)\lambda_{\bar{\alpha}} & \text{ for }\alpha=\bar{\alpha}
	\end{cases} \\
	\text{and} \quad \nu^{(2)}_\alpha=  & \begin{cases}
	\lambda_\alpha-\tau\theta_\alpha\lambda_{\bar{\alpha}} & 
	\text{ for }\alpha \in (\lambda^+ \setminus \{ \bar{\alpha} \}) \cup \{\beta\},\\
	(1+\tau)\lambda_{\bar{\alpha}} & \text{ for }\alpha=\bar{\alpha}
	\end{cases}
	\end{align*}
	(and 0 outside of $\lambda^+ \cup \{\beta\}$)
	are non-proportional
	elements of $N_\beta$ with $(\nu^{(i)})^+\subset \lambda^+$ for $i=1,2$.
	Moreover,
	$\cA\nu^{(i)}=  \cA\lambda$ for $i=1,2$ and 
	$\lambda\in\relint[\nu^{(1)},\nu^{(2)}]$,
	which contradicts the $X$-circuit property of $\lambda$.

For the second statement,
suppose $\beta 	\in \conv(\lambda^+)$ and $\cA\lambda\ne \mathbf{0}$. Then,
there exists a normalized element $\lambda' \in N_\beta$ with 
$\lambda^+=(\lambda')^+$ and  $\cA\lambda'=\mathbf{0}$.
	Let $\tau$ be the maximal real number such that
	$\nu^{(1)} := \lambda - \tau \lambda' \in N_{\beta}$.
	That maximum clearly exists,
	and, since $(\lambda')^+ =\lambda^+$, the number 
	$\tau$ is positive.
	Moreover, since $\lambda$ and $\lambda'$ are normalized, we
	have $\tau\le 1$. 

	The sublinear circuit
	$\nu^{(2)} := 
	\lambda + \tau\lambda' $ is clearly contained in $N_{\beta}$ as well.
	Since $\lambda, \lambda'$ are non-proportional and $\tau > 0$,
	the sublinear circuits
	$\nu^{(1)}$ and $\nu^{(2)}$ are non-proportional.
	Further, since $\nu^{(1)} + \nu^{(2)} = 2\lambda$,
	we see that $\lambda$ 
	can be written as a convex combination of the two non-proportional
	elements $\nu^{(1)} \in N_{\beta}$ 
	and $\nu^{(2)} \in N_{\beta}$.	
	Due to $\mathcal{A} \lambda' =\mathbf{0}$, we obtain
	$
		\sigma_X(-\mathcal{A} \nu^{(1)})  = 
		\sigma_X(-\mathcal{A} \nu^{(2)}) =\sigma_X(-\mathcal{A} \lambda)
  $
  and thus
  \[
  \sigma_X(-\mathcal{A} \lambda)= \frac{1}{2}(\sigma_X(-\mathcal{A} \nu^{(1)})  + 
		\sigma_X(-\mathcal{A} \nu^{(2)})).
  \]
Hence, $\lambda\notin\Lambda_{X}(\cA,\beta)$.
\end{proof}

We can provide the following two cases of the converse direction of Theorem \ref{th:convex_hull_lambda}. In particular, both cases will be applicable
for $X=[-1,1]^n$. We can assume that $\beta \in \conv(\lambda^+)-\rec(X)^*$
since otherwise any $\lambda \in N_{\beta} \setminus\{\mathbf{0}\}$ 
will have $\sigma_X(-\cA\lambda) = \infty$ and hence, violate condition (1)
in the definition of an $X$-circuit.

\begin{lemma}\label{prop:X_circuits_edge_cases}
  Given $\beta \in \cA$, let $\lambda\in N_\beta \setminus \{\mathbf{0}\}$ 
	be normalized with $\beta\in \conv(\lambda^+)-\rec(X)^*$ and 
	such that $\lambda^+$ consists of affinely independent vectors.
	\begin{enumerate}
		\item If $|\supp \lambda|=2$ or
		\item if $X$ is full-dimensional,
		$\beta\in\conv(\lambda^+)$, $\cA\lambda=\mathbf{0}$,
	\end{enumerate}
	then $\lambda\in\Lambda_X(\cA,\beta)$.
\end{lemma}

Note that, since in the theorem $\lambda^+$ consists of affinely independent vectors, we have $\relint\conv(\lambda^+)\cap\lambda^+=\emptyset$. 

\begin{remark} If the property of full-dimensionality is omitted in the second condition,
the statement is not true anymore. As a counterexample, let $X$ be the singleton set
$X=\{1\}$ and let
$\mathcal{A} = \{1,2,3\}$. Then $\lambda = \frac{1}{2}(1,-2,1)^T$ is not an $X$-circuit, because
$\lambda = \frac{1}{2}\lambda^{(1)} + \frac{1}{2} \lambda^{(2)}$ with 
$\lambda^{(1)} = (1,-1,0)^T$
and $\lambda^{(2)} = (0,-1,1)^T$ and $\nu \to \sigma_X(-\cA \nu)$ is linear on
$[ \lambda^{(1)}, \lambda^{(2)} ]$. Note that the functions
$x \mapsto \exp(\alpha^T x)$, $\alpha \in \cA$ are not linearly independent on $X$.
\end{remark}
\begin{proof}
	For the first statement, suppose there exist $\nu^{(1)},\nu^{(2)}\in N_\beta$ 
	decomposing $\lambda$. Then $\supp(\nu^{(i)})\subseteq\supp \lambda$
	for $i \in \{1,2\}$, because the cancellation of terms not contained in 
	$\supp \lambda$ is not possible, as the negative term always corresponds to $\beta$. Since $\nu^{(1)}_{\beta} <0$ and 
	$\nu^{(2)}_{\beta} < 0$ and $|\supp \lambda| = 2$, both $\nu^{(1)}$ and
	$\nu^{(2)}$ are proportional to $\lambda$.
	
	Now consider the second condition. Since the property of being an $X$-circuit
	is invariant under translation of $X$, we can assume without loss of
	generality that $\mathbf{0}\in\inter X$.
	Suppose that there exist non-proportional,
	normalized $\lambda^{(1)},\lambda^{(2)}\in N_\beta$ and $\theta_1,\theta_2\in (0,1)$ with $\theta_1+\theta_2=1$ such that
	\begin{align*}
	\sum\limits_{i=1}^2 \theta_i(\lambda^{(i)}, \sigma_X(-\cA\lambda^{(i)}))=(\lambda,\sigma_X(-\cA\lambda)).
	\end{align*}
	We distinguish two cases. 
	If $\mathcal{A} \lambda^{(1)} = \mathbf{0}$, then $\mathcal{A} \lambda^{(2)} = 
	-\frac{\theta_1}{\theta_2}\cA \lambda^{(1)} = \mathbf{0}$. Hence, the uniqueness
	of the barycentric coordinates with respect to a given affinely independent ground set
	implies $\lambda^{(1)} = \lambda^{(2)}$, which is a contradiction to their
	non-proportionality.
	
	If $\mathcal{A} \lambda^{(1)} \neq \mathbf{0}$, then, as the argument above states that $\cA\lambda^{(2)}=\mathbf{0}$ implies  $\cA\lambda^{(1)}=\mathbf{0}$, we have 
	$\mathcal{A} \lambda^{(2)} = - \frac{\theta_1}{\theta_2}\mathcal{A} \lambda^{(1)} \neq \mathbf{0}$ as well.
	Then $\mathbf{0} \in \inter X$ implies
	$\sigma_X(-\cA \lambda^{(1)}) > 0$ and $\sigma_X(-\cA\lambda^{(2)}) > 0$.	
	Since $\sigma_X(-\cA \lambda) = -\sigma_X(\mathbf{0}) = 0$, the mapping
	$\nu \mapsto \sigma_X(-\cA \nu)$ cannot be linear on 
	$[\lambda^{(1)}, \lambda^{(2)}]$.

\end{proof}

\subsection*{\texorpdfstring{$X$}{X}-circuits of polyhedral 
  cones \texorpdfstring{$X$}{X}}

As discussed after Proposition~\ref{prop:polyhedron_x_finite_circuits}, in the case
of polyhedral cones $X$
we always have $\sigma_X(-\cA\lambda)=0$ whenever this value is finite. 
Since we will reduce the determination of the sublinear circuits $\Lambda_X(\cA)$
for a cone $X$ in some prominent cases to the classical affine circuits
$\Lambda_{\R^n}(\cA)$ (which of course is also a case of a polyhedral cone), 
we first look at an example for the latter case.

In the following, we examine sublinear circuits for various sets $X\subset\R^n$ (for some $n\in\N$) and support sets of the form $\mathcal{A} = \{(i,j) \ : \ 1 \le i,j \le k\}$, $k\in\N$. In these situations, we can write a sublinear circuit $\nu$ as a matrix $M^{(\nu)}\in\R^{k\times k}$ such that $M^{(\nu)}_{i,j}=\nu_{(i,j)}$ 
for all $(i,j) \in \cA$.

\begin{example}
\label{ex:circuits-r2}
For $X=\R^2$ and support $\mathcal{A} = \{(i,j) \ : \ 1 \le i,j \le 3\}$,
there are 16 sublinear circuits (up to
multiples). Namely, there are 8 sublinear circuits
with support size~3 (all of them have non-zero entries $1,-2,1$; they appear
in the three rows, the three columns and the two diagonals of the $3 \times 3$-matrix).
Moreover, there are the following 8 sublinear circuits of support size~4.
Here, the upper left entry of the matrices refers to the support point $(1,1)$:
\begin{equation}
  \label{eq:r2-supp4}
   \begin{array}{cccccc}
    \left( \begin{array}{rrr}
      1 &  0 & 1 \\
      0 &  -4 & 0 \\
      0 &  2 & 0
    \end{array} \right),
   &
      \left( \begin{array}{rrr}
      0 &  1 & 0 \\
      1 &  -3 & 0 \\
      0 &  0 & 1
    \end{array} \right)
  \end{array}
\end{equation}
as well as the 90-degree, 180-degree and 270-degree rotations about the
$(2,2)$-element of these matrices.
As $\mathbf{0}\in\interior\R^2$ and $\rec(\R^2)^*=\{\mathbf{0}\}$, this 
reflects in particular the statements of Theorem~\ref{th:convex_hull_lambda} and Lemma~\ref{prop:X_circuits_edge_cases}.
\end{example}

Next we consider the sublinear circuits of the non-negative orthant
$\R_+^n$. For a non-empty subset $S \subset [n]$ and a support point 
$\alpha \in \cA \subset \R^n$, we write $\alpha_S$ for the projection of
$\alpha$ onto the components of $S$, i.e., $\alpha_S := (\alpha_{s})_{s \in S}$.
We also set
$\cA_S := \{ \alpha_S \, : \, \alpha \in \cA\}$ and for a matrix $M$ with
$n$ rows, we set $M_S$ as the submatrix of $M$ defined
by the rows with indices in $S$, which in particular yields $M_S\lambda=(M\lambda)_S$.

\begin{theorem}\label{th:xcircuit_nonnegative_orthant}
Let $n \ge 2$ and $X=\R_+^n$ and $\beta \in \cA$. A normalized element $\lambda\in N_\beta$ with $|\lambda^+| \ge 2$
is contained in $\Lambda_X(\cA, \beta)$ if and only if 
there exists a non-empty subset $S \subset[n]$ with
$|\{\alpha_S \, : \, \alpha \in \supp \lambda\}| = |\supp \lambda|$
such that
$\lambda$ is an $\R^{|S|}$-circuit for the support set $\mathcal{A}_S$
and $(\cA \lambda)_{[n] \setminus S} > \mathbf{0}$.
\end{theorem}

\begin{remark}
The latter condition in Theorem~\ref{th:xcircuit_nonnegative_orthant}
implies $\beta_S = (\cA \lambda^+)_S$ and, hence, 
$\beta_S \in \relint \conv((\lambda^+)_S)$, and
$\beta_{[n] \setminus S} \in \conv ((\lambda^+)_{[n] \setminus S}) -\R_+^{{[n] \setminus S}}$.
\end{remark}

\begin{proof}[Proof of Theorem~\ref{th:xcircuit_nonnegative_orthant}]
Let $\lambda \in \Lambda_X(\cA,\beta)$ 
with $|\lambda^+| \ge 2$.
Hence, $\cA \lambda \ge \mathbf{0}$.
For every $s \in [n]$ with $(\cA \lambda)_{\{s\}} > 0$, we observe that
$\lambda$ is also an $\R^{n-1}_+$-circuit for 
$\cA_{[n]\setminus \{s\}}$.
The $X$-circuit property
of $\lambda$ and $|\lambda^+| \ge 2$
imply that there exists at least one $s \in [n]$
with $(\cA \lambda)_{\{s\}} = 0$; otherwise, choosing a 
vector $\nu$
supported on a two-element subset of $\lambda^+$ with entries 
$\varepsilon$ and $-\varepsilon$ for sufficiently small 
$\varepsilon > 0$ would give a non-trivial decomposition
$\lambda = 
(\frac{1}{2} \lambda - \nu) + (\frac{1}{2} \lambda + \nu)$.

Let $S$ be the inclusion-maximal subset $S \subset [n]$ with $(\cA \lambda)_S=\mathbf{0}$.
By the initial considerations, $S \neq \emptyset$ and 
$\lambda$ is an $\R^{|S|}$-circuit of $\cA_S$. This implies
the cardinality statement
$|\{\alpha_S \, : \, \alpha \in \supp \lambda\}| = |\supp \lambda|$.
By definition of $S$, we have $(\cA \lambda)_{[n] \setminus S} > \mathbf{0}$.

Conversely, let
$\emptyset \neq S \subset [n]$ with
$|\{\alpha_S \, : \, \alpha \in \supp \lambda\}|  = |\supp \lambda|$ 
such that
$\lambda$ is an $\R^{|S|}$-circuit of $\mathcal{A}_S$
and $(\cA \lambda)_{[n] \setminus S} > \mathbf{0}$. 
Then $\lambda$ is an $\R_+^{|S|}$-circuit for $\cA_S$ and, further, 
an $X$-circuit for $\cA$.
\end{proof}

Theorem~\ref{th:xcircuit_nonnegative_orthant} can be used in the
reduction of the enumeration of all $X$-circuits to the
enumeration of all classical affine circuits.

\begin{example}
For $X=\R_+^2$ and the support set
$\mathcal{A} = \{(i,j) \ : \ 1 \le i,j \le 3\}$,
there are 65 normalized sublinear circuits. 
Namely, by Theorem~\ref{th:xcircuit_nonnegative_orthant}, there are
\begin{enumerate}
\item 27 normalized sublinear circuits of cardinality~2:
  $\lambda = - e^{(i_1,j_1)} + e^{(i_2,j_2)}$
    for $1 \le i_1 \le i_2 \le 3$, $1 \le j_1 \le j_2 \le 3$;
  that is, the entry ``$1$'' appears in ``lower right'' quadrant of the entry ``$-1$''.
\item 16 normalized sublinear circuits in which the entries $\frac{1}{2},-1,\frac{1}{2}$ 
  appear in columns 1,2,3, respectively, 
  such that the entry $-1$ appears above the line through the two entries $\frac{1}{2}$.
  \item 16 normalized sublinear circuits in which the entries $\frac{1}{2}$   appear in rows 1,2,3, respectively,
  such that the $-1$ appears left to the line containing the two entries $\frac{1}{2}$.
\item 8 $\R^2$-circuits of cardinality 4, which are the 
normalized versions of the ones from 
  Example~\ref{ex:circuits-r2}.
\end{enumerate}
Since the diagonal and the anti-diagonal are counted both in cases (2) and (3),
we have to subtract 2, which gives $27+16+16+8-2 = 65$.
The following table shows in row $i$ and column $j$ the number of sublinear
circuits with $\nu^- = \{(i,j)\}$.

\begin{center}
\begin{tabular}{|r||r|r|r|} \hline
     & 1 & 2 & 3 \\ \hline \hline
 1 & 8 & 14 & 2 \\ \hline
 2 & 14 & 21 & 2 \\ \hline
 3 & 2 & 2 & 0 \\ \hline
 \end{tabular}
 \end{center}

Exemplarily, for the case $\nu^- = \{(1,2)\}$, there are five circuits of type (1)
as well as the following nine (in the subsequent list not normalized) 
sublinear circuits $\nu$ with $\nu^- = \{(1,2)\}$, i.e., the component
with index $(1,2)$ is the negative component. As before, the upper left entry of the
matrices refer to the support point $(1,1)$:

\[
   \begin{array}{cccccc}

      \left( \begin{array}{rrr}
      1 & -2 & 1 \\
      0 &  0 & 0 \\
      0 &  0 & 0
    \end{array} \right),
 &
      \left( \begin{array}{rrr}
      1 & -2 & 0 \\
      0 &  0 & 1 \\
      0 &  0 & 0
    \end{array} \right),
 &
      \left( \begin{array}{rrr}
      1 & -2 & 0 \\
      0 &  0 & 0 \\
      0 &  0 & 1
    \end{array} \right),
 &
      \left( \begin{array}{rrr}
      0 & -2 & 1 \\
      1 &  0 & 0 \\
      0 &  0 & 0
    \end{array} \right),
 &
      \left( \begin{array}{rrr}
      0 & -2 & 0 \\
      1 &  0 & 1 \\
      0 &  0 & 0
    \end{array} \right),
 \\ [3.5ex]
       \left( \begin{array}{rrr}
      0 & -2 & 0 \\
      1 &  0 & 0 \\
      0 &  0 & 1
    \end{array} \right),
 &
      \left( \begin{array}{rrr}
      0 & -2 & 1 \\
      0 &  0 & 0 \\
      1 &  0 & 0
    \end{array} \right),
 &
      \left( \begin{array}{rrr}
      0 & -2 & 0 \\
      0 &  0 & 1 \\
      1 &  0 & 0
    \end{array} \right),
 &
      \left( \begin{array}{rrr}
      0 & -2 & 0 \\
      0 &  0 & 0 \\
      1 &  0 & 1
    \end{array} \right).
 \end{array}
 \]

\end{example}

The following theorem characterizes the connection between the $X$-circuits 
and the $\R^n$-circuits for more general polyhedral cones $X$.

\begin{theorem}
\label{th:connect-to-rn-1}
Let $X = \pos \{v^{(1)}, \ldots, v^{(k)}\}$
be an $n$-dimensional polyhedral cone spanned by the
vectors $v^{(1)}, \ldots, v^{(k)}$, where $k \ge n$.
 Then
\begin{equation}
  \label{eq:connect-to-rn-1}
  \left\{ \lambda \in \Lambda_X(\mathcal{A}) \ : \ \mathcal{A} \lambda = \mathbf{0}
   \right\}  
   \ = \ \Lambda_{\R^n}(\mathcal{A}).
\end{equation}
\end{theorem}

\begin{proof}Fix $\beta \in \cA$ and denote by $W$ the $k \times n$-matrix
whose rows are the transposed vectors $(v^{(1)})^T, \ldots, (v^{(k)})^T$. Hence,
$X^* = \{x \in \R^n \ : \ Wx \ge {\bf 0} \}$.  
The set $\Lambda_X(\cA,\beta)$ is the set of normalized vectors spanning the
extreme rays of the cone
\begin{align*}
   K_X & \ = \ \{ \nu \in N_{\beta} \ : \ \sigma_X(-\cA \nu) \le \mathbf{0} \} 
     \ = \ \{ \nu \in N_{\beta} \ : \ \cA \nu \in X^*\} \\
    & \ = \ \{ \nu \in N_{\beta} \, : \, W{\cA}\nu \ge {\bf 0} \}
\end{align*}
and the set $\Lambda_\R^n(\cA,\beta)$ is the set of normalized vectors spanning the
extreme rays of the cone
\begin{align*}
  K_{\R^n} & \ = \ \{ \nu \in N_{\beta} \ : \ \sigma_{\R^n}(-\cA \nu) \le \mathbf{0}\} 
   \ = \ \{ \nu \in N_{\beta} \ : \ \cA \nu = \mathbf{0} \}.
\end{align*}
Since the matrix $W$ has rank $n$, the linear mapping $x \mapsto Wx$ is 
injective, and thus its kernel is $\{\mathbf{0}\}$. Hence, 
$K_{\R^n} = \{ \nu \in N_{\beta} \ : \ W \cA \nu = \mathbf{0} \}$.
The cone $K_{\R^n}$ is contained in the cone $K_X$. As a consequence,
if $\lambda \in N_{\beta}$
is not contained in the right hand side of~\eqref{eq:connect-to-rn-1}, 
it is not contained in the left hand side.

Conversely, let $\lambda \in N_{\beta}$ be contained in the 
right hand side of~\eqref{eq:connect-to-rn-1}. Then $\cA \lambda = {\bf 0}$
and $W \cA \lambda = {\bf 0}$.
Assume there exists a decomposition into a convex combination
$\lambda = \theta_1 \lambda^{(1)} + \theta_2 \lambda^{(2)}$ 
with $W \cA \lambda^{(1)} \neq \mathbf{0}$.
Since $W \cA \lambda = {\bf 0}$ and $W \cA \lambda^{(1)} \ge {\bf 0}$, 
at least one component
of $W\cA (\lambda - \theta_1 \lambda^{(1)}) = W \cA \theta_2 \lambda^{(2)}$ is smaller 
than zero. 
This is a contradiction.
Hence, $\lambda$ is contained in the left hand side of~\eqref{eq:connect-to-rn-1}.
\end{proof}

\section{The \texorpdfstring{$n$}{n}-dimensional 
cube \texorpdfstring{$X=[-1,1]^n$}{}\label{se:cube}}
We discuss the sublinear circuits of the $n$-dimensional cube $[-1,1]^n$, which is a 
prominent case of a compact polyhedron.
Throughout the section, we assume $X=[-1,1]^n$ for some fixed $n\in\N$ and 
$\cA\subset\R^n$ non-empty and finite. 
We can already apply some of the former statements to gain knowledge of the structure of $X$-circuits. For example, as $\rec(X)^*=\R^n=-\rec(X)^*$, Lemma \ref{prop:X_circuits_edge_cases} implies that every element supported on exactly two points is an $X$-circuit. Hence, we examine the structure of those $X$-circuits
$\lambda\in \Lambda_X(\cA)$ that have more than two support points. We begin
with a necessary criterion.

\begin{lemma}
    \label{le:critminus1to1}
	Let $\lambda\in  N_\beta$ with $\lambda_\beta=-1$ for some $\beta\in\cA$ and
	$|\supp \lambda|\ge 3$. If for all $j\in[n]$ 
	\begin{align}
	\label{eq:critminus1to1}
	\left( \alpha_j\le\beta_j \text{ for all }\alpha\in\lambda^+ \right)
	\: \text{ or } \:
	\left( \alpha_j\ge\beta_j \text{ for all }\alpha\in\lambda^+ \right),
	\end{align} then $\lambda\notin\Lambda_X(\cA)$.
\end{lemma}
Note that the precondition expresses that there exists a vertex $v$ of $[-1,1]^n$
such that for all $\alpha \in \lambda^+$, the maximal face of the function
$x \mapsto (\beta- \alpha)^T x$
contains $v$.

\begin{proof}
    We can assume 
    $\beta\notin\relint(\conv(\lambda^+))$, since otherwise the preconditions
    imply $\beta = \alpha$ for all $\alpha\in \lambda^+$, violating 
    $|\supp \lambda|\ge 3$. Hence, we have $\cA \lambda \neq \mathbf{0}$ and
    the supremum of $x \mapsto (-\cA \nu)^T x$ is attained at some vertex
    of $[-1,1]^n$.
    
    Now assume $\lambda \in \Lambda_X(\cA)$. In order to come up with a contradiction,
    we construct a decomposition of $\lambda = \sum_{\alpha \in \lambda^+} \nu^{(\alpha)}$
    with supports $\supp\{\nu^{(\alpha)}\} = \{\alpha,\beta\}$ of cardinality~2
    by setting
    	\begin{align*}
	\theta_\alpha\nu_\alpha^{(\alpha)}:=\lambda_\alpha \; \text{ and } \; \theta_\alpha\nu_\beta^{(\alpha)}:=-\theta_\alpha\nu_\alpha^{(\alpha)}=-\lambda_\alpha\text{ for all }\alpha\in\lambda^+.
	\end{align*}
	We observe that $\nu^{(\alpha)} \in N_\beta$ for all $\alpha \in \lambda^+$ and $(\theta_\alpha)_{\alpha\in\lambda^+}$ can be chosen with the property $\sum_{\alpha \in \lambda^{+}} \theta_{\alpha} = 1$.
	Moreover, $\sum\limits_{\alpha\in\lambda^+}\theta_\alpha\nu^{(\alpha)}= \lambda$ and 
	\begin{align*}
	\sum\limits_{\alpha\in\lambda^+}\theta_\alpha\sigma_X(-\cA\nu^{(\alpha)}) \ = \ &  \sum\limits_{\alpha\in\lambda^+}\theta_\alpha\sum\limits_{j=1}^n\left|\nu_\alpha^{(\alpha)}(\alpha_j-\beta_j)\right|
	=  \sum\limits_{\alpha\in\lambda^+}\sum\limits_{j=1}^n\left|\lambda_\alpha(\alpha_j-\beta_j)\right|\\
	\ \overset{(\ref{eq:critminus1to1})}{=} & \ \sum\limits_{j=1}^n\left|\sum\limits_{\alpha\in\lambda^+} \lambda_\alpha(\alpha_j-\beta_j)\right| = \sigma_X(-\cA\lambda).
	\end{align*}
  By distinguishing the cases $\alpha_j = \beta_j$ and $\alpha_j \neq \beta_j$, it is 
  straightforward to see that this expression in terms of a convex combination is
  locally linear. Hence, $\lambda$ cannot be an $X$-circuit, which is the contradiction.
\end{proof}

We provide a slightly more general version of Lemma~\ref{le:critminus1to1}, 
whose proof is analogous.

\begin{lemma}
	Let $\lambda\in  N_\beta$ with $\lambda_\beta=-1$ for some $\beta\in\cA$, 
	and $|\supp \lambda|\ge 3$. 
	Further suppose that for $
	J(\lambda):=\{j:\beta_j=\sum_{\alpha\in\cA}\lambda_\alpha \alpha_j\}$,
	 the support can be disjointly decomposed into the two sets 
	\begin{align*}
		\cA^{(1)}=\{\alpha: \alpha_j=\beta_j \text{ for all }j\notin J(\lambda)\} \: \text{ and } \:
		\cA^{(2)}= \{\alpha:\alpha_j= \beta_j \text{ for all }j\in J(\lambda)\} \ne \emptyset
	\end{align*}
	such that for all $j\in[n]\setminus J(\lambda)$ we have
	\[
	  \left(
	   \alpha_j\le\beta_j \text{ for all }\alpha\in \cA^{(2)}
	   \right) 
	   \: \text{ or } \:
	   \left( \alpha_j\ge\beta_j \text{ for all }\alpha\in \cA^{(2)}
	   \right) .
	\] 
	Then $\lambda$ is not an $X$-circuit of $\cA$.
\end{lemma}

\begin{example}\emph{The planar case $[-1,1]^2$.} For the case of the planar
square $X=[-1,1]^2$ we provide some explicit descriptions of the sublinear 
circuits for support sets located on a grid $\{(i,j) \, : \, 1 \le i, j \le k\}$
for some $k \in \N$.

If $\lambda$ is a normalized $[-1,1]^2$-circuit, then, due to
$\rec([-1,1]^2)^*=-\rec([-1,1]^2)^*=\R^2$, there is no restriction on the location
of the negative coordinate. However, using Theorem~\ref{th:convex_hull_lambda}, we can exclude potential sublinear circuits $\lambda\in N_\beta$ for some $\beta\in\cA$, where $\relint\conv(\supp \lambda)\cap\lambda^+\ne\emptyset$ and those where $\beta\in\conv(\lambda^+)$ but $\cA\lambda\ne \mathbf{0}$; in particular, the latter situation
excludes 
the case $\beta\in\conv(\lambda^+)\setminus\relint\conv(\lambda^+)$. 
Moreover, using Lemma \ref{le:critminus1to1}, we can exclude all those potential $[-1,1]^2$-circuits where $|\supp \lambda|\ge 3 $ and $
	\left( \alpha_j\le\beta_j \text{ for all }\alpha\in\lambda^+ \right)$
or 
	$\left( \alpha_j\ge\beta_j \text{ for all }\alpha\in\lambda^+ \right)$.

For the case $k=3$, i.e., the support set
$\mathcal{A} = \{ (i,j) \, : \, 1 \le i, j \le 3\}$, the structural statements 
facilitate
to obtain the exact set of sublinear circuits. Up to multiples, 
there are 132 $X$-circuits:
\begin{enumerate}
\item 72 sublinear circuits supported on two elements: $e^{(i_1,j_1)} - e^{(i_2,j_2)}$
  for $1 \le i_1,i_2,j_1,j_2 \le 3$ with $(i_1,j_1) \neq (i_2,j_2)$.
\item 27 sublinear circuits in which the entries $\frac{1}{2},-1,\frac{1}{2}$
  appear in columns 1,2,3, respectively.
\item 27 sublinear circuits in which the entries $\frac{1}{2},-1,\frac{1}{2}$
  appear in rows 1,2,3, respectively.
\item 8 sublinear circuits supported on 4 elements.
\end{enumerate}
Since the diagonal and the anti-diagonal are counted both in cases 2 and 3, this
gives 72+27+27+8-2 = 132 sublinear circuits.
The following table shows in row $i$ and column $j$ the number of normalized sublinear
circuits $\lambda$ with $\lambda^- = \{(i,j)\}$.

\begin{center}
\begin{tabular}{|r||r|r|r|} \hline
     & 1 & 2 & 3 \\ \hline \hline
 1 & 8 & 17 & 8 \\ \hline
 2 & 17 & 32 & 17 \\ \hline
 3 & 8 & 17 & 8 \\ \hline
 \end{tabular}
 \end{center}

The subsequent list gives the 17
(not necessarily normalized)
$X$-circuits $\nu$ with $\nu^{-} = \{(1,2)\}$, i.e., the component with index
$(1,2)$ is the negative component. As before, the upper left entry of the 
matrices refer to the support point $(1,1)$:
\[
   \begin{array}{cccccc}
    \left( \begin{array}{rrr}
      1 & -1 & 0 \\
     0 &  0 & 0 \\
      0 &  0 & 0
    \end{array} \right),
   &
      \left( \begin{array}{rrr}
      0 & -1 & 0 \\
      1 &  0 & 0 \\
      0 &  0 & 0
    \end{array} \right),
 &
      \left( \begin{array}{rrr}
      0 & -1 & 0 \\
      0 &  0 & 0 \\
      1 &  0 & 0
    \end{array} \right),
    &
      \left( \begin{array}{rrr}
      0 & -1 & 0 \\
      0 &  1 & 0 \\
      0 &  0 & 0
    \end{array} \right) ,
 &
      \left( \begin{array}{rrr}
      0 & -1 & 0 \\
      0 &  0 & 0 \\
      0 &  1 & 0
    \end{array} \right), \\ [3.5ex]
      \left( \begin{array}{rrr}
      0 & -1 & 1 \\
      0 &  0 & 0 \\
      0 &  0 & 0
    \end{array} \right),
 &
      \left( \begin{array}{rrr}
      0 & -1 & 0 \\
      0 &  0 & 1 \\
      0 &  0 & 0
    \end{array} \right),
 &
      \left( \begin{array}{rrr}
      0 & -1 & 0 \\
      0 &  0 & 0 \\
      0 &  0 & 1
    \end{array} \right) ,
    &
      \left( \begin{array}{rrr}
      1 & -2 & 1 \\
      0 &  0 & 0 \\
      0 &  0 & 0
    \end{array} \right), 
    &
        \left( \begin{array}{rrr}
      0 & -2 & 1 \\
      1 &  0 & 0 \\
      0 &  0 & 0
    \end{array} \right), 
    \\ [3.5ex]
        \left( \begin{array}{rrr}
      0 & -2 & 1 \\
      0 &  0 & 0 \\
      1 &  0 & 0
    \end{array} \right),
     &
      \left( \begin{array}{rrr}
      1 & -2 & 0 \\
      0 &  0 & 1 \\
      0 &  0 & 0
    \end{array} \right), 
    &
        \left( \begin{array}{rrr}
      0 & -2 & 0 \\
      1 &  0 & 1 \\
      0 &  0 & 0
    \end{array} \right) , 
    &
        \left( \begin{array}{rrr}
      0 & -2 & 0 \\
      0 &  0 & 1 \\
      1 &  0 & 0
    \end{array} \right), 
    &
      \left( \begin{array}{rrr}
      1 & -2 & 0 \\
      0 &  0 & 0 \\
      0 &  0 & 1
    \end{array} \right) ,
    \\ [3.5ex]
        \left( \begin{array}{rrr}
      0 & -2 & 0 \\
      1 &  0 & 0 \\
      0 &  0 & 1
    \end{array} \right), 
    &
        \left( \begin{array}{rrr}
      0 & -2 & 0 \\
      0 &  0 & 0 \\
      1 &  0 & 1
    \end{array} \right).
  \end{array}
\]

\noindent
\emph{The case $k=4$.} In the case $\mathcal{A} = \{ (i,j) \, : \, 1 \le i, j \le 4\}$, a computer calculation shows that
there are 980 normalized $X$-circuits, which come in the following classes with regard to
$\lambda^-$:
\begin{center}
\begin{tabular}{|r||r|r|r|r|} \hline
     & 1 & 2 & 3 & 4 \\ \hline \hline
 1 & 15 & 47 & 47 & 15 \\ \hline
 2 & 47 & 136 & 136 & 47 \\ \hline
 3 & 47 & 136 & 136 & 47 \\ \hline
 4 & 15 & 47 & 47 & 15 \\ \hline
 \end{tabular}
 \end{center}

Note that in this case, the criteria of this and the previous section are not sufficient to determine the set of sublinear circuits solely from these criteria.
\end{example}

\section{Reducibility and extremality\label{se:reducibility}}

By Proposition~\ref{prop:decomp2},
the reduced sublinear circuits 
provide an irredundant decomposition of conditional SAGE cones.
In this section, we discuss some criteria and key examples for reduced 
sublinear circuits. As an application of the criteria, we will determine
the extremals of the $[-1,1]$-SAGE cone in 
Theorem \ref{extreme_ray_one_dim_polytope}.

For the classical case of affine circuits supported on a finite set $\mathcal{A}$, 
the following exact characterization in terms of the support is known.

\begin{proposition}\emph{(\cite[Cor.\ 4.7]{knt-2020}, \cite[Thm.~3.2]{FdW-2019})}
\label{pr:combinatorial-rn2}
A vector $\nu$ is a reduced $\R^n$-circuit if and only if
\[
  \mathcal{A} \cap \relint \conv \nu^+ = \{ \nu^-\}.
\]
\end{proposition}

For example, with regard to the two matrices in~\eqref{eq:r2-supp4}
of Example \ref{ex:circuits-r2}, the left one is not reduced, but the right one
is.
The following theorem gives a generalization for the necessary direction of
Proposition~\ref{pr:combinatorial-rn2} to the constrained situation,
where $X$ is a non-empty, convex set in $\R^n$.

\begin{theorem}\label{th:combinatorial-reducibility1}
	Let $\lambda \in \Lambda_X(\cA,\beta)$.
	If there exists $\beta'\in\cA\setminus \supp \lambda$ and some normalized  $\lambda'\in N_{\beta'}$ where $(\lambda')^+ \subset \supp(\lambda)$ 
and $\cA\lambda' = \gamma \cA\lambda$ for some $\gamma\ge 0$,
then $\lambda$ is not reduced.
\end{theorem}

Before providing the proof within this section, we 
discuss its consequences.

\begin{corollary}\label{cor:convex_hull_empty}
	Let $\lambda \in \Lambda_X(\cA,\beta)$. 
	If $(\conv(\supp \lambda)\cap\cA)\setminus \supp \lambda \ne\emptyset$,
	then $\lambda$ is not reduced. Consequently,
	\begin{align*}
		\{\lambda\in\Lambda_X^\star(\cA):\cA\lambda=\mathbf{0}\}\subset \Lambda_{\R^n}^\star(\cA).
	\end{align*}
\end{corollary}
\begin{proof}
	The first statement follows by applying Theorem \ref{th:combinatorial-reducibility1} with $\beta'\in\conv(\supp \lambda)\setminus\supp\lambda$, $(\lambda')^+$ are the vertices of  $\supp \lambda$
and $\gamma=0$. The second one is a direct consequence of Proposition \ref{pr:combinatorial-rn2} and the fact that for $X=\R^n$ all $X$-circuits $\lambda$ have the property $\cA\lambda=\mathbf{0}$.
\end{proof}

Using this corollary, we can provide an analogon to Theorem~\ref{th:connect-to-rn-1}.

\begin{theorem}
	\label{th:connect-reduced-to-rn-1}
	Let $X = \pos \{v^{(1)}, \ldots, v^{(k)}\}$
	be an $n$-dimensional polyhedral cone spanned by the
	vectors $v^{(1)}, \ldots, v^{(k)}$, where $k \ge n$.
	Then
	\begin{equation}
	\label{eq:connect-reduced-to-rn-1}
	\left\{ \lambda \in \Lambda_X^\star(\mathcal{A}) \ : \ \mathcal{A} \lambda = 0 \right\}  
	\ = \ \Lambda^\star_{\R^n}(\mathcal{A}).
	\end{equation}
\end{theorem}

\begin{proof}
By Corollary \ref{cor:convex_hull_empty}, every $\lambda\in\Lambda_X^\star(\cA)$ is contained in $\Lambda_{\R^n}^\star(\cA)$. Suppose there exists some $\lambda\in\Lambda_{\R^n}^\star(\cA)$ that is not contained in $\Lambda_X^\star(\cA)$. By Theorem \ref{th:connect-to-rn-1}, $\lambda\in\Lambda_X(\cA)$. As $\lambda\notin\Lambda_X^\star(\cA)$, there exist $m\in\N$ and
$X$-circuits $\nu^{(1)},\ldots,\nu^{(m)}$ which are non-proportional to $\lambda$
	and which satisfy
	$\sum_{i\le m}(\nu^{(i)},\sigma_X(- \cA \nu^{(i)}) )= 
	(\lambda,\sigma_X(- \cA \lambda))$. Since $\sigma_X(-\cA\lambda)=0$ and $\sigma_X(y) \in \{0,\infty\}$ for all $y\in\R^n$, we have
	$\sigma_X(- \cA \nu^{(i)}) =0$ for all $i\in[m]$.
	
	As in Theorem \ref{th:connect-to-rn-1}, denote by $W$ the $k \times n$-matrix
	whose rows are the transposed vectors $(v^{(1)})^T, \ldots, (v^{(k)})^T$. Again,
	\begin{align*}
	\sigma_X(-y) <  \infty \text{ if and only if } Wy\ge {\bf 0}.
	\end{align*}
	Since $\cA\lambda=W\cA\lambda=\mathbf{0}$, we obtain 
	$W{\cA}\nu^{(i)} = {\bf 0}$ and, as the kernel of $W$ is $\{\mathbf{0}\}$, further
	${\cA}\nu^{(i)} = {\bf 0}$ for all $i\in[m]$. Hence, 
	$\nu^{(i)}\in \Lambda_{\R^n}^\star(\cA)$ and therefore $\lambda\notin\Lambda_{\R^n}^\star(\cA)$, which is a contradiction.
\end{proof}

We illustrate the
applicability of Theorem~\ref{th:combinatorial-reducibility1}
in determining the reduced sublinear circuits by returning to the univariate example $X=[-1,1]$, 
which was started in Proposition \ref{prop:circuits-univariate2}.

\begin{thm}\label{thm:extremalXcircuits}
	Let $X=[-1,1]$ and $\cA=\{\alpha_1,\ldots,\alpha_m\}$ sorted ascendingly,
	where $m \ge 3$. Then, $\Lambda_X^\star(\cA)$ consists of the
	following sublinear circuits:
	\begin{enumerate}
		\item $\lambda=e^{(2)}-e^{(1)}$ or $\lambda=e^{(m-1)}-e^{(m)}$,
		or 
		\item $\displaystyle{\lambda=
		\frac{\alpha_{i-1}-\alpha_{i}}{\alpha_{i-1}-\alpha_{i+1}}
		e^{(i-1)}- e^{(i)}+
		\frac{\alpha_{i-1}-\alpha_{i}}{\alpha_{i-1}-\alpha_{i+1}}
		e^{(i+1)}}$
		for some $i\in\{2,\ldots, m-1\}$.
    \end{enumerate}
\end{thm}

Note that this gives, in particular,
\begin{align*}
\Lambda_X^\star(\cA)\cap \left\{\lambda\in \bigcup\limits_{\beta\in\cA} N_\beta: |\supp \lambda|=3\right\} \ = \ \Lambda_\R^\star(\cA).
\end{align*}

\begin{proof}[Proof of Theorem \ref{thm:extremalXcircuits}]
	By Proposition \ref{prop:circuits-univariate2} and Corollary \ref{cor:convex_hull_empty}, the only candidates for normalized
	reduced $X$-circuits are 
	\begin{enumerate}
		\item $\lambda=e^{(i)}-e^{(i\pm 1)}$ or
		\item $\displaystyle{\lambda=
		\frac{\alpha_i-\alpha_{i+1}}{\alpha_{i-1}-\alpha_{i+1}}
		e^{(i-1)}- e^{(i)} +
		\frac{\alpha_{i-1}-\alpha_{i}}{\alpha_{i-1}-\alpha_{i+1}}
		e^{(i+1)}}$ for some $i\in\{2,\ldots, m-1\}$.

	\end{enumerate}
	For every $X$-circuit $e^{(i+1)}-e^{(i)}$ with $i > 1$, the $X$-circuit $e^{(i+1)}-e^{(1)}$ satisfies the precondition of Theorem \ref{th:combinatorial-reducibility1} and for every $X$-circuit
	$e^{(i-1)}-e^{(i)}$ with $i < m$, the $X$-circuit $e^{(i-1)}-e^{(m)}$ satisfies the precondition of Theorem \ref{th:combinatorial-reducibility1}. Hence, all those $X$-circuits are not reduced. 
	
	We see that for all $i\in[m]$, there is precisely one normalized $X$-circuit $\lambda$ that appears in the listed set of \emph{possible} reduced $X$-circuits. As $\rec(X)^*=\R$, there exists at least one $X$-AGE signomial where the $i$-th coefficient is negative, hence 
 $C_X(\cA,\alpha_i)\ne\emptyset$ for all $i\in[m]$.
 As $C_X(\cA,\alpha_i)$ is the union of several $\lambda$-witnessed 
 $X$-AGE cones and those cones can be solely represented by reduced $X$-circuits (compare \cite{mnt-2020}, Sections $4$ and $5$), for every $i\in[m]$ there exists at least one reduced 
 $X$-circuit in $C_X(\cA,\alpha_i)$.
With this, the statement follows.
\end{proof}

\begin{proof}[Proof of Theorem~\ref{th:combinatorial-reducibility1}]
		Since $\lambda$ and $\lambda'$ are normalized elements in $N_\beta$ and $N_{\beta'}$, we have
	\[
	\begin{array}{rcll@{}l}
	\sum_{\alpha \in \lambda^+} \lambda_{\alpha} & = & 1 \text{ and }
	& \lambda_{\beta} = -1, \; & \lambda_{\alpha} \ge 0 \text{ for } \alpha \in 
	\mathcal{A} \setminus \{\beta\}, \\
	\sum_{\alpha \in (\lambda')^+} \lambda'_{\alpha} & = & 1 \text{ and }
	& \lambda'_{\beta'} = -1, \; & \lambda'_{\alpha} \ge 0 \text{ for } \alpha \in 
	\mathcal{A} \setminus \{\beta'\}.
	\end{array}
	\]

	Let $\tau$ be the maximal real number in $[0,1/\gamma]$ (with 
	the convention $1/\gamma:=\infty$ if $\gamma =0$) 
	such that
	$\nu^{(1)} := \lambda - \tau \lambda' \in N_{\beta}$.
	That maximum clearly exists,
	and, since $(\lambda')^+ \subset \supp \lambda$, the number 
	$\tau$ is positive.
	Moreover, since $\lambda$ and $\lambda'$ are normalized and distinct, we
	have $\tau  < 1$. 
	
	Similarly, let $\tau'$ be the maximal real number 
	in $[0,\gamma]$ with
	$\nu^{(2)} := 
\lambda' - \tau' \lambda \in N_{\beta'}$. Here, we have $0 \le \tau' \le 1$ 
(and, in particular, $\tau'=0$ if $\gamma=0$ or $(\lambda')^+\subsetneq\lambda^+$).
	Hence, $\nu^{(1)} \in N_{\beta},\nu^{(2)} \in N_{\beta'}$ and $1-\tau\tau'\in(0,1]$.
	
	Since $\nu^{(1)} + \tau \nu^{(2)} = \lambda - \tau \lambda' +\tau\lambda' - \tau\tau' \lambda
	=(1 - \tau \tau') \lambda$,
	we see that $\lambda$ 
	can be written as a conic combination of the two non-proportional
	(not necessarily normalized) elements $\nu^{(1)} \in N_{\beta}$ 
	and $\nu^{(2)} \in N_{\beta'}$.
	Due to $\mathcal{A} \lambda' = \gamma\mathcal{A} \lambda$ and as both, $1-\tau\gamma\ge 0$ and $\gamma-\tau'\ge 0$, we obtain
	\begin{eqnarray*}
		\sigma_X(-\mathcal{A} \nu^{(1)}) & = &
		\sigma_X(-\mathcal{A} \lambda + \tau\mathcal{A} \lambda')=\sigma_X(-\mathcal{A} \lambda + \tau\gamma\mathcal{A} \lambda)\\
		&	= &  (1-\tau\gamma)\sigma_X(-\mathcal{A} \lambda )
		= \sigma_X(-\mathcal{A} \lambda) - \tau \sigma_X(-\mathcal{A} \lambda'), \\
		\sigma_X(-\mathcal{A} \nu^{(2)}) & = &
		\sigma_X(-\mathcal{A} \lambda' + \tau' \mathcal{A} \lambda) =	\sigma_X(-\gamma\mathcal{A} \lambda + \tau' \mathcal{A} \lambda) 
		\\
	&	= & (\gamma-\tau')\sigma_X(-\mathcal{A} \lambda)
		=\sigma_X(-\mathcal{A} \lambda') - \tau' \sigma_X(-\mathcal{A} \lambda)
	\end{eqnarray*}
	and further
	\begin{eqnarray*}
		\sigma_X(-\mathcal{A}\lambda)
		& = & \frac{1}{1 - \tau \tau'}
		\left(
		\sigma_X \left( - \cA \lambda \right)
		- \tau
		\sigma_X \left( - \cA \lambda' \right)
		+ \tau
		\sigma_X \left( - \cA \lambda' \right)
		- \tau \tau'
		\sigma_X \left( - \cA \lambda \right) \right) \\
		& = &
		\frac{1}{1 - \tau \tau'} \left(
		\sigma_X \left( - \cA \nu^{(1)} \right)
		+ 
		\tau \sigma_X \left( - \cA \nu^{(2)} \right)
		\right),
	\end{eqnarray*}
	which shows that $(\lambda,\sigma_X(-\mathcal{A}\lambda))$ does not
	generate an extreme ray in $G_X(\mathcal{A})$.
	By definition of a reduced sublinear circuit, $\lambda \in \Lambda_X^\star(\cA)$.

\end{proof}

As a consequence of the results in this section, we can give an exact 
characterization of the extreme rays of the $[-1,1]$-SAGE cone.

\begin{thm}\label{extreme_ray_one_dim_polytope}
	Let $X=[-1,1]$ and $\cA=\{\alpha_1,\ldots,\alpha_m\}$ be sorted ascendingly,
	where $m \ge 3$. The extremal rays of $C_X(\cA)$ are the following:
	\begin{enumerate}
		\item $\R_+\cdot\left( \exp(\alpha_2x)-\exp(\alpha_{1}-\alpha_2) 
		\exp(\alpha_1x) \right)$,
		\item $\R_+\cdot\left( \exp(\alpha_{m-1}x)-\exp(\alpha_{m-1}-\alpha_{m})
		\exp(\alpha_{m}x) \right)$,
		\item $\R_+\cdot \{c_{i-1}\exp(\alpha_{i-1}x)+c_{i}\exp(\alpha_{i}x)+c_{i+1}\exp(\alpha_{i+1}x)\}$, with
		\begin{align*}
		& 	c_{i-1}>0, \; c_{i+1}>0 \text{ and }c_i=-\left(\frac{c_{i-1}}{\lambda_{i-1}}\right)^{\lambda_{i-1}}\left(\frac{c_{i+1}}{\lambda_{i+1}}\right)^{\lambda_{i+1}}, \text{ where} \nonumber \\
		& \lambda_{i-1}=\frac{\alpha_{i+1}-\alpha_{i}}{\alpha_{i+1}-\alpha_{i-1}},
		\, \lambda_{i+1}=\frac{\alpha_{i}-\alpha_{i-1}}{\alpha_{i+1}-\alpha_{i-1}}
		\text{ and } 
		\alpha_{i-1} - \alpha_{i+1} \le \ln \frac{c_{i-1}\lambda_{i+1}}{c_{i+1}\lambda_{i-1}}
		\le \alpha_{i+1} - \alpha_{i-1}.
		\end{align*}
	\end{enumerate}
\end{thm}

We first deal with the atomic extreme rays, that is, extreme rays which are supported
on a single element. These extreme rays are not captured by the $X$-circuit view.

\begin{lemma}[Atomic extreme rays of $C_X(\cA)$ for compact sets $X$]\label{prop:atomicextremerays}
	Let $X\subset \R^n$ be a compact set 
	and $\cA\subset \R^n$ finite with $|\cA|\ge 2$. Then, there are no atomic extreme rays of $C_X(\cA)$.
\end{lemma}

\begin{proof} As in Lemma \ref{prop:X_circuits_edge_cases}, we use invariance of the $X$-circuits under translation of $X$ and can w.l.o.g.\ assume ${\bf 0}\in X$. Let $\alpha \neq \beta\in\cA$ arbitrary. 	
	Assume that $f=c_{\alpha}\exp(\alpha^Tx)$ with $c_{\alpha}>0$
	is extremal. 
	We observe that $\lambda\in  N_\beta$ with $\lambda_\alpha=1=-\lambda_\beta$ is an $X$-circuit inducing the ray 
	\begin{align*}
	\R_+\cdot \left(\exp(\alpha^Tx)-\frac{1}{\exp(s)}\exp(\beta^Tx)\right),
	\end{align*}
	where $s\ge 0$ is finite and such that
	$
	\sigma_X(-\cA\lambda))=s.
	$
	Hence,  the $X$-AGE signomials
	\begin{align*}
	f^{(1)}= c_\alpha \exp(\alpha^Tx) -\frac{c_\alpha}{\exp(s)} \exp(\beta^Tx), f^{(2)}=\frac{c_\alpha}{\exp(s)}\exp(\beta^Tx)
	\end{align*}
	sum to $f$, contradicting the extremality of $f$.
\end{proof}

\begin{proof}[Proof of Theorem \ref{extreme_ray_one_dim_polytope}]
	Let $\cA = \{\alpha_1,\ldots,\alpha_m\}$ be sorted ascendingly.
	By Lemma \ref{prop:atomicextremerays}, there are no atomic extreme rays,
	and by Theorem \ref{thm:extremalXcircuits} and Proposition \ref{prop:decomp2},
	all the extreme rays are supported on two or three elements.
	
	We start by considering the 2-term case. By
	Lemma~\ref{thm:extremalXcircuits}, the only candidates for the
	extreme rays are the ones given in the cases (1) and (2). Since these
	cases are symmetric, it suffices to consider case $(1)$, i.e.,  $f(x) = \exp(\alpha_2 x) - \exp(\alpha_1-\alpha_2)\exp(\alpha_1 x)$. Any conic combination of 3-term
	AGE functions and of functions of case (2) has a lowest-exponent term with
	positive coefficient. Hence,
	$f$ cannot be written as a convex 
	combination of 3-term AGE 
	functions and of functions of case~(2). Thus, $f$ indeed is extremal.
	
	Now consider the 3-term case. By 
	Lemma~\ref{thm:extremalXcircuits}, the only candidates for extreme
	rays are of the form
	$f(x)=c_{i-1}\exp(\alpha_{i-1}x)+c_{i}\exp(\alpha_{i}x)+
	c_{i+1}\exp(\alpha_{i+1}x)$ with $c_{i-1} > 0$, $c_{i+1} > 0$
	and $c_i < 0$.
	The proof in \cite[Theorem 6.1]{mnt-2020} shows that $f$ must have 
	a zero in $[-1,1]$ and that the location $x^*$ of the zero is
	\[
	x^* \ = \ \ln \left( \frac{c_{i-1} \lambda_{i+1}}{c_{i+1} \lambda_{i-1}}
	\right) / (\alpha_{i+1} - \alpha_{i-1}),
	\]
	where $\lambda_{i-1}$ and $\lambda_{i+1}$ are defined as in case (3)
	of the theorem. This gives the defining condition for $c_i$ as well as the inequality
	conditions in case (3).
	
	Any decomposition of $f$ cannot involve a 2-term AGE function. For
	$x^* \in (-1,1)$, this follows from the strict positivity of the 2-term AGE
	functions of type (1) and (2). For the boundary situations 
	$x^* \in \{-1,1\}$, we can additionally
	use the derivative condition $f'(x^*) = 0$ to exclude the 
	2-term AGE functions.
	
	It remains to show that the 3-term AGE function
	$f$ cannot be decomposed in terms of 3-term
	AGE functions. However, since $f$ has a zero in $[-1,1]$ and thus in $\R$,
	it induces an extremal ray of the cone $C_{\R}(\cA)$ and cannot be decomposed
	using only 3-term AGE functions by \cite[Proposition 4.4]{knt-2020}.

\end{proof}

\begin{example} \emph{The reduced sublinear circuits for the cube $[-1,1]^2$.}
We consider again the support $\mathcal{A} = \{ (i,j) \, : \, 1 \le i, j \le k \}$
for some $k \in \N$.
In the case $k=3$, there are 24 normalized reduced $X$-circuits, which
come in the following classes:
\begin{enumerate}
\item 12 sublinear circuits with entries $1,-1$, namely,
\begin{enumerate}
  \item 8 with entry $-1$ in a corner and entry $+1$ beside or below the corner,
  \item 4 with entry $-1$ in a non-corner boundary entry and entry 
    $+1$ in the central, interior entry,
\end{enumerate}
\item 8 sublinear circuits, where the sequence 
  $\frac{1}{2},-1,\frac{1}{2}$ appears in a 
  row (3 possibilities), in a column (3 possibilities) or on the diagonal or
  the antidiagonal,
\item 4 sublinear circuits supported on 4 elements, namely
\[
    \left( \begin{array}{rrr}
      0 &  1/3 & 0 \\
      1/3 &  -1 & 0 \\
      0 &  0 & 1/3
    \end{array} \right)
\]
as well as the 90-degree, 180-degree and 270-degree rotation of thesis
matrix.

Note that, when starting from the set of all sublinear circuits $\lambda$
for $[-1,1]^2$, Theorem~\ref{th:combinatorial-reducibility1} is applicable to rule out 
that $\lambda$ is reduced in a number of cases. For example, the matrices
\[
   \begin{array}{cccccc}
    \left( \begin{array}{rrr}
      0 &  0 & 1/2 \\
      1/2 &  -1 & 0 \\
      0 &  0 & 0
    \end{array} \right),
   &
      \left( \begin{array}{rrr}
      0 & 0 & 1/2 \\
      1/2 &  0 & 0 \\
      0 &  -1 & 0
    \end{array} \right)
  \end{array}
\]
represent sublinear circuits $\lambda$ and $\lambda'$ with 
$\cA \lambda = (-1/2,0)^T$ and $\cA \lambda' = (-3/2,0)^T$, 
to which Theorem~\ref{th:combinatorial-reducibility1} 
can be applied in order to show that $\lambda$ is not reduced.

Also note that all reduced $\R^2$-circuits for the support set $\mathcal{A}$ 
are also reduced $[-1,1]^2$-circuits. Namely, since for all other 
$[-1,1]^2$-circuits $\lambda$, we have $\sigma_X(-\cA \lambda) \neq \mathbf{0}$,
those circuits cannot be used to decompose an $\R^2$-circuit (which
has $\sigma_X(-\cA \lambda) = \mathbf{0}$).
\end{enumerate}

In the case $k=4$ with 16 support points, a computer calculation shows 
that there are 72 reduced sublinear circuits.
\end{example}

\section{Conclusion and outlook\label{se:conclusion}}

We have studied the connection of sublinear circuits and their supports
and the sublinear circuits for polyhedral sets $X$.
Since for polyhedral sets, the number of $X$-circuits is finite, this allows
to apply polyhedral and combinatorial techniques. In particular the
$X$-SAGE cones can be decomposed into a finite number of power 
cones, which arise from the reduced sublinear circuits.

For non-polyhedral sets $X$, in general the number of $X$-circuits is not 
finite anymore. It remains a future task to study necessary and sufficient
criteria for sublinear circuits of structured non-polyhedral sets, such as sets with 
symmetry; for recent work
on symmetric SAGE-based optimization see \cite{symmetry-amgm}.
In a different direction,
Forsg{\aa}rd and de Wolff \cite{FdW-2019} have characterized
the boundary of the SAGE cone through a connection between circuits
and tropical geometry. It also remains for future work to establish a
generalization of this, aiming at connecting the conditional SAGE cone and sublinear circuits
to tropical geometry.

\bibliography{bib-sublinear}
\bibliographystyle{plain}

 \end{document}